\documentclass{amsart}

\usepackage{graphicx}

\newcommand{\mc}[1]{\mathcal{#1}}
\newcommand{\mb}[1]{\mathbb{#1}}
\newcommand{\mr}[1]{\mathrm{#1}}
\newcommand{\mf}[1]{\mathfrak{#1}}
\newcommand{\on}[1]{\operatorname{#1}}

\newcommand{\hypa}{{\mbox{$_2${\bf F}$_1$}}}

\newcommand{\tx}[1]{\text{#1}}

\theoremstyle{remark}
\newtheorem{exa}{Example}
\newtheorem{rem}{Remark}
\theoremstyle{plain}
\newtheorem{lem}{Lemma}
\newtheorem{thm}[lem]{Theorem}
\newtheorem{cor}[lem]{Corollary}
\newtheorem{pro}[lem]{Proposition}

\date{26th November 2014}
\title{Basis properties of the $p,q$-sine functions} 
\author{Lyonell Boulton}
\author{Gabriel J. Lord}
\email{L.Boulton@hw.ac.uk \quad and \quad G.J.Lord@hw.ac.uk}
\address{Department of Mathematics and Maxwell Institute for Mathematical
Sciences, Heriot-Watt University, Edinburgh, EH14 4AS, UK}

\begin{document}

\begin{abstract}
We improve the currently known thresholds for basisness of the family of periodically 
dilated $p,q$-sine functions. Our findings rely on a Beurling  decomposition 
of the corresponding change of coordinates in terms of shift operators of 
infinite multiplicity. We also determine refined bounds on the Riesz constant 
associated to this family. These results  seal mathematical gaps 
in the existing  literature on the subject.
\end{abstract}

\maketitle

\tableofcontents


\newpage

\section{Introduction}
Let $p,q>1$.  Let $F_{p,q}:[0,1]\longrightarrow [0,\pi_{p,q}/2]$ be the integral
\[
     F_{p,q}(y)=\int_0^y \frac{\mr{d} x}{(1-x^q)^{\frac1p}} 
\]
where $\pi_{p,q}=2F_{p,q}(1)$. The $p,q$-\emph{sine functions},  
$
\sin_{p,q}:\mb{R}\longrightarrow [-1,1],
$
are defined to be the inverses of $F_{p,q}$,
\[
\sin_{p,q}(x)=F^{-1}_{p,q}(x) \qquad  \text{for all} \qquad x\in[0,\pi_{p,q}/2]
\]
extended to $\mb{R}$ by the rules
\[
      \sin_{p,q}(-x)=-\sin_{p,q}(x) \qquad \text{and} \qquad
      \sin_{p,q}(\pi_{p,q}/2-x)= \sin_{p,q}(\pi_{p,q}/2+x),
\]
which make them periodic, continuous, odd with respect to 0 and even with respect to $\frac{\pi_{p,q}}{2}$.
These are natural generalisations of the sine function, indeed 
\[\sin_{2,2}(x)=\sin(x) \qquad \text{and} \qquad \pi_{2,2}=\pi,\]
 and they are known to share a number of remarkable properties with their classical counterpart \cite{Lindqvist1995,EdmundsLang2011}. 

 Among these properties lies the fundamental question of completeness and linear independence of the family $\mc{S}=\{s_n\}_{n=1}^\infty$ where $s_n(x)=\sin_{p,q}(\pi_{p,q}nx)$. This question has received some attention recently \cite{EdmundsGurkaLang2012,BushellEdmunds2012,EdmundsLang2011, BBCDG2006}, with a particular emphasis 
 on the case $p=q$. In the latter instance, $\mathcal{S}$ is the set of eigenfunctions of the generalised eigenvalue problem for the one-dimensional $p$-Laplacian subject to Dirichlet boundary conditions \cite{BindingDrabek2003,BL2010}, which is known to be of relevance in the theory of slow/fast diffusion processes, \cite{EFG1997}. 
See also the related papers \cite{EdmundsGurkaLang2014-2,EdmundsGurkaLang2014-1}.
   
Set $e_n(x)=\sqrt{2}\sin(n\pi x)$, so that $\{e_n\}_{n=1}^\infty$ is a Schauder basis of the Banach space $L^r\equiv L^r(0,1)$ for all $r>1$.  The family $\mc{S}$ is also a Schauder basis of $L^r$ if and only if the corresponding \emph{change of coordinates map}, $A:e_n\longmapsto s_n$, extends to a linear homeomorphism of $L^r$. 
The Fourier coefficients of $s_n(x)$ associated to $e_k$ obey the relation
\begin{align*}
    \widehat{ s_n}(k) &= \int_0^1 s_1(n x) e_k(x) \mr{d}x\\
         &=\sum_{m=1}^\infty \widehat{s_1}(m)  \int_0^1 e_{mn}(x) e_k(x) \mr{d}x = \left\{ 
         \begin{aligned}  \widehat{s_1}(m) & \quad \text{if $mn=k$ for some $m\in \mb{N}$} \\
         0 & \quad \text{otherwise.} \end{aligned} \right.
\end{align*}
For $j\in \mb{N}$, 
let \[a_j\equiv a_j(p,q)=\widehat{s_1}(j) = \sqrt{2} \int_0^1
\sin_{p,q}(\pi_{p,q}x)\sin(j\pi x) \mr{d}x
\] 
(note that $a_j=0$ for $j\equiv_2 0$) and let $M_j$ be the linear
isometry such that $M_je_k=e_{jk}$. Then 
\begin{align*}
  Ae_n&=s_n=\sum_{k=1}^\infty \widehat{s_n}(k)e_k=\sum_{j=1}^\infty \widehat{s_1}(j) e_{jn} 
= \left( \sum_{j=1}^\infty  a_j M_j  \right) e_n,
\end{align*}
so that the change of coordinates takes the form
\begin{equation} \label{decoA}
      A=\sum_{j=1}^\infty a_j M_j.
 \end{equation}

Notions of ``nearness'' between bases of Banach spaces are known to play a fundamental role in classical mathematical analysis, \cite[p.265-266]{Kato1967}, \cite[\S I.9]{Singer1970} or \cite[p.71]{Higgins1977}. Unfortunately, the expansion \eqref{decoA} strongly suggests that $\mc{S}$ is not globally ``near'' $\{e_n\}_{n=1}^\infty$, e.g. in the Krein-Lyusternik or the Paley-Wiener  sense, \cite[p.106]{Singer1970}. Therefore classical arguments, such as those involving the Paley-Wiener Stability Theorem, are unlikely to be directly applicable in the present context.

In fact, more rudimentary methods can be invoked in order to examine the invertibility of the change of coordinates map. From \eqref{decoA} it follows that 
\begin{equation} \label{trick1}
      \sum_{j=3}^\infty |a_j| <|a_1| \quad \Rightarrow \quad \left\{\begin{aligned} &A,A^{-1}\in \mc{B}(L^r) \\
      & \|A\|\|A^{-1}\|\leq \frac{\sum_{j=1}^\infty |a_j|}{|a_1|-\sum_{j=3}^\infty |a_j|}. \end{aligned} \right. 
\end{equation}
In \cite{BBCDG2006} it was claimed that the left side of \eqref{trick1} held true for all $p=q\geq p_1$ where $p_1$ was determined to lie in the segment $\left(1,\frac{12}{11}\right)$. Hence  $\mc{S}$ would be a Schauder basis, whenever $p=q\in (p_1,\infty)$. 

Further developments in this respect were recently reported by Bushell and Edmunds \cite{BushellEdmunds2012}. These authors cleverly fixed a gap originally published in \cite[Lemma~5]{BBCDG2006} and observed that, as the left side of \eqref{trick1} ceases to hold true whenever  
\begin{equation}   \label{break2}
      a_1= \sum_{j=3}^\infty a_j,
\end{equation}
the argument will break for $p=q$ near $p_2\approx 1.043989$.  Therefore, the basisness question for $\mc{S}$ should be tackled by different means in the regime $p,q\to 1$.

More recently \cite{EdmundsGurkaLang2012}, Edmunds, Gurka and Lang, employed \eqref{trick1} in order to show invertibility of $A$ for general pairs $(p,q)$, as long as
 \begin{equation} \label{constrpq}
     \pi_{p,q}<\frac{16}{\pi^2-8}.
 \end{equation}
Since \eqref{constrpq} is guaranteed whenever
\begin{equation} \label{constrpq2}
      \frac{p}{q(p-1)}<\frac{4}{\pi^2-8},
\end{equation}
this allows $q\to 1$ for $p>\frac{4}{12-\pi^2}$. 
However, note that a direct substitution of $p=q$ in \eqref{constrpq2}, only leads to the sub-optimal condition $p>\frac{\pi^2}{4}-1\approx 1.467401$. 

In Section~\ref{linearind} below we show that the family $\mc{S}$ is $\omega$-\emph{linearly independent} for all $p,q>1$, see Theorem~\ref{likernelandspan}. In Section~\ref{ribap} we establish conditions ensuring that $A$ is a homeomorphism of $L^2$ in a neighbourhood of the region in the $(p,q)$-plane where
\[
      \sum_{j=3}^\infty |a_j|=a_1,
\]
see Theorem~\ref{inprovement} and also Corollary~\ref{beyonda}. For this purpose, in Section~\ref{criteria} we find two further criteria which generalise \eqref{trick1} in the Hilbert space setting, see corollaries~\ref{main_1} and \ref{main_2}. In this case, the \emph{Riesz constant}, 
\[
    r(\mc{S})=\|A\| \|A^{-1}\|
\]
characterises how $\mc{S}$ deviates from being an orthonormal basis. These new statements yield upper bounds for $r(\mc{S})$, which improve upon those obtained from the right side of \eqref{trick1}, even when the latter is applicable.

The formulation of the alternatives to \eqref{trick1} presented below relies crucially on work developed in Section~\ref{toep_s}. From  Lemma~\ref{multareshifts} we compute explicitly the Wold decomposition of the isometries $M_j$: they turn out to be shifts of infinite multiplicity. Hence we can extract from the expansion   \eqref{decoA} suitable components  which are Toeplitz operators of scalar type acting on appropriate Hardy spaces. As the theory becomes quite technical for the case $r\not=2$ and all the estimates analogous to those reported below would involve a dependence on the parameter $r$, we have chosen to restrict our attention with regards to these improvements only to the already interesting Hilbert space setting.

Section~\ref{casep=q} is concerned with particular details of the case
of equal indices $p=q$, and it involves results on both the general
case $r>1$ and the specific case $r=2$. Rather curiously, we have
found another gap which renders incomplete the proof of invertibility
of $A$ for $p_1<p<2$  originally published in \cite{BBCDG2006}. See
Remark~\ref{rem_gap}. 
Moreover, the application of \cite[Theorem~4.5]{BushellEdmunds2012} only gets to a \emph{basisness threshold} of $\tilde{p}_1\approx 1.198236>\frac{12}{11}$, where $\tilde{p}_1$ is defined by the identity
\begin{equation}   \label{bueva}
     \pi_{\tilde{p}_1,\tilde{p}_1}=\frac{2\pi^2}{\pi^2-8}.
\end{equation}
See also \cite[Remark~2.1]{EdmundsLang2011}. 
In Theorem~\ref{fixingBBCDG} we show that $\mc{S}$ is indeed a Schauder basis of $L^r$ for $p=q\in(p_3,\frac65)$ where $p_3\approx 1.087063<\frac{12}{11}$, see \cite[Problem~1]{B2012}.  As $\frac65>\tilde{p}_1$, basisness is now guaranteed for all $p=q>p_3$. See Figure~\ref{impro_fig_p=q}.

In Section~\ref{nume} we report on our current knowledge of the different thresholds for invertibility of the change of coordinates map, both in the case of equal indices and otherwise. Based on the new criteria found in Section~\ref{criteria}, we formulate a general test of invertibility for $A$ which is amenable to analytical and numerical investigation. This test involves finding sharp bounds on the first few coefficients $a_k(p,q)$. See Proposition~\ref{beyond2}. For the case of equal indices, this test indicates that $\mc{S}$ is a Riesz basis of $L^2$ for $p=q>p_6$ where $p_6\approx 1.043917<p_2$.  

All the numerical quantities reported in this paper are accurate up to the last digit shown, which is rounded to the nearest integer. In the appendix we have included fully reproducible computer codes which can be employed to verify the calculations reported.


\section{Linear independence} \label{linearind}

A family $\{\tilde{s}_n\}_{n=1}^\infty$ in a Banach space is called $\omega$-linearly independent \cite[p.50]{Singer1970}, if 
\[\sum_{n=1}^{\infty} f_n \tilde{s}_n =0 \qquad \Rightarrow \qquad
f_n=0 \text{ for all } n.
\]   

\begin{thm} \label{likernelandspan}
For all $p,q>1$, the family $\mc{S}$ is $\omega$-linearly independent in $L^r$. Moreover, if the linear extension of the map $A:e_n\longmapsto s_n$ is a bounded operator $A:L^2\longrightarrow L^2$, then
\[
      \left(\overline{\on{Span} \mc{S} }\right)^\perp = \on{Ker} A^*.
\]
\end{thm}
\begin{proof}
For the first assertion we show that $\text{Ker}(A)=\{0\}$. Let $f=\sum_{k=1}^\infty f_k e_k$ be such that $Af=0$ where the series is convergent in the norm of $L^r$. Then
\[
      \sum_{j=1}^\infty \left(\sum_{mn=j} f_ma_n\right) e_j=\sum_{jk=1}^\infty f_ka_je_{jk}=0.
\]
Hence
\begin{equation} \label{kertequalzero}
 \sum_{mn=j} f_m a_n=0 \qquad \forall j\in \mb{N}.
\end{equation}
We show that all $f_j=0$ by means of a double induction argument.

Suppose that $f_1\not= 0$. We prove that all $a_k=0$. Indeed, clearly $a_1=0$ from \eqref{kertequalzero} with $j=1$. Now assume inductively that $a_j=0$ for all $j=1,\ldots,k-1$. From \eqref{kertequalzero} for $j=k$ we get
\[
       0=f_1a_k+\sum_{\substack{mn=k \\ m\not=1 \ n\not=k}}f_ma_n=f_1a_k.
\]
Then $a_k=0$ for all $k\in \mb{N}$. As this would contradict the fact that $A\not=0$, necessarily $f_1=0$.

Suppose now inductively that $f_1,\ldots,f_{l-1}= 0$ and $f_l\not=0$. We prove that again all $a_k=0$.  Firstly, $a_1=0$
from \eqref{kertequalzero} with $j=l$, because
\[
     0=f_la_1 +\sum_{\substack{mn=l \\ m\not=l \ n\not=1}}f_ma_n=f_l a_1.
\]
Secondly, assume by induction that $a_j=0$ for all $j=1,\ldots, k-1$. From \eqref{kertequalzero} for $j=lk$ 
we get
\[
       0=f_la_k + \sum_{\substack{mn=lk \\ m\not=l \ n\not=k}}f_ma_n=f_la_k.
\]
The latter equality is a consequence of the fact that, for $mn=lk$ with $m\not=l$ and $n\not=k$, either
$m<l$ (indices for the $f_m$) or $n<k$ (indices for the $a_n$). Hence $a_k=0$ for all $k\in \mb{N}$.  As this would again contradict the fact that  $A\not=0$, necessarily all $f_k=0$ so that $f=0$.

The second assertion is shown as follows. Assume that $A\in \mc{B}(L^2)$. If $f\in \on{Ker}
A^*$, then $\langle f,Ag\rangle=0$ for all $g\in L^2$, so $f\perp
\on{Ran}A$ which in turns means that $f\perp s_n$ for all $n\in
\mb{N}$. On the other hand, if the latter holds true for $f$, then
$f\perp A e_n$ for all $n\in \mb{N}$, so $A^*f=0$, as required.
\end{proof}

Therefore, $\mc{S}$ is a Riesz basis of $L^2$ if and only if $A\in \mc{B}(L^2)$ and $\on{Ran}A=L^2$. A simple example illustrates how a family of dilated periodic functions  can break its property of being a Riesz basis. 

\begin{exa} \label{ex1} Let $\alpha\in [0,1]$. Take
\begin{equation}   \label{littleexample}
    \tilde{s}(x)=\frac{1-\alpha}{\sqrt{2}}\sin(\pi x)+ \frac{\alpha}{\sqrt{2}}\sin(3 \pi x).
\end{equation}
By virtue of Lemma~\ref{toep} below, $\tilde{\mc{S}}=\{\tilde{s}(nx)\}_{n=1}^\infty$ is a Riesz basis
of $L^2$ if and only if $0\leq \alpha<\frac12$. For $\alpha=1$ we have
an orthonormal set. However it is not complete, as it clearly misses the infinite-dimensional subspace $\on{Span}\{e_{j}\}_{j\not \equiv_3 0}$.
\end{exa}


\section{The different components of the change of coordinates map}
\label{toep_s}

The fundamental decomposition of $A$ given in \eqref{decoA} allows us
to extract suitable components formed by Toeplitz operators of scalar type, \cite{RosenblumRovnyak1984}.    In order to identify these components, we begin by determining the Wold decomposition of the isometries $M_j$, \cite{RosenblumRovnyak1984,Nikolskii1986}.  See Remark~\ref{diri}.

\begin{lem} \label{multareshifts}
For all $j>1$, $M_j\in \mc{B}(L^2)$ is a shift of infinite multiplicity.
\end{lem}
\begin{proof}
Define 
\begin{align*} 
\mc{L}_0^j& =\on{Span}\{e_k\}_{k\not\equiv_j 0}=\on{Ker}(M_j^*) \qquad \tx{and} \\
\mc{L}_n^j&= M_j^n\mc{L}_0^j \qquad \tx{for} \qquad n\in \mb{N}.
\end{align*}
Then $\mc{L}_n^j\cap \mc{L}_m^j=\{0\}$ for $m\not=n$, $L^2=\bigoplus_{n=0}^\infty \mc{L}_n^j$, and $M_j:\mc{L}_{n-1}^j\longrightarrow \mc{L}_{n}^j$ one-to-one and onto for all $n\in\mb{N}$.  Therefore indeed $M_j$ is a shift of multiplicity $\dim \mc{L}_0^j=\infty$.
\end{proof}

Let $\mb{D}=\{|z|< 1\}$.  The Hardy spaces of functions in $\mb{D}$ with values in the Banach space $\mc{C}$ are denoted below by $H^\gamma(\mb{D};\mc{C})$.  Let
\[
    \tilde{b}(z)=\sum_{k=0}^\infty b_k z^k
\]
be a holomorphic function on $\overline{\mb{D}}$ and fix $j\in \mb{N}\setminus \{1\}$. Let 
\[
\tilde{B}\in H^\infty(\mb{D};\mc{B}(\mc{L}_0^j)) \qquad \text{be given by} \qquad \tilde{B}(z)=\tilde{b}(z)I.
\]
Let  the corresponding Toeplitz operator \cite[(5-1)]{RosenblumRovnyak1984}
\[
      T(\tilde{B})\in \mc{B}(H^2(\mb{D};\mc{L}_0^j)) \qquad \text{be given by} \qquad T(\tilde{B}):f(z)\mapsto \tilde{B}(z)f(z).
\]
Let 
\begin{equation} \label{opeB} 
B= \sum_{k=0}^\infty b_k M_{j^k}:L^2 \longrightarrow L^2.
\end{equation}
By virtue of Lemma~\ref{multareshifts} (see  \cite[\S3.2 and \S5.2]{RosenblumRovnyak1984}), there exists an invertible isometry \[U:L^2\longrightarrow  H^2(\mb{D};\mc{L}_0^j)\] such that $UB=T(\tilde{B})U$. 
Below we write
\[
    \mf{M}(\tilde{b})=\max_{z\in \overline{\mb{D}}} |\tilde{b}(z)| \qquad \text{and} \qquad 
    \mf{m}(\tilde{b})=\min_{z\in \overline{\mb{D}}} |\tilde{b}(z)|.
\]

\begin{thm} \label{generic_toep}
$B$ in \eqref{opeB} is invertible if and only if $\mf{m}(\tilde{b})>0$.
Moreover
\[\|B\|= \mf{M}(\tilde{b})\qquad \text{and} \qquad \|B^{-1}\|=\mf{m}(\tilde{b}) ^{-1}.
\]
\end{thm}
\begin{proof}
Observe that $T(\tilde{B})$ is scalar analytic in the sense of \cite[\S3.9]{RosenblumRovnyak1984}.
 Since $\tilde{b}$ is holomorphic in $\overline{\mb{D}}$, then $\mf{M}(\tilde{b})<\infty$ and
\[
     \|B\|=\|T(\tilde{B})\|=\|\tilde{B}\|_{H^\infty(\mb{D};\mc{B}(\mc{L}_0^j))}=\mf{M}(\tilde{b})
\]
\cite[\S4.7 Theorem~A(iii)]{RosenblumRovnyak1984}.

If $0\not\in \tilde{b}(\overline{\mb{D}})$, then $\tilde{b}(z)^{-1}$ is also holomorphic in $\overline{\mb{D}}$. The scalar Toeplitz operator $T(\tilde{b})$ is invertible if and only if
$\mf{m}(\tilde{b})>0$. Moreover, \cite[\S1.5]{BottcherSilvermann1998}, 
\[
      T(\tilde{b})^{-1}=T(\tilde{b}^{-1}) \in \mc{B}(H^2(\mb{D};\mb{C})).
\]
 The matrix of $T(\tilde{B})$ has the block representation \cite[\S5.9]{RosenblumRovnyak1984}
\[
    T(\tilde{B}) \sim \begin{bmatrix}  b_0 I & 0 & 0 &\cdots \\ \\
                                       b_1 I & b_0 I & 0 & \cdots \\ \\
                                        b_2 I & b_1 I & b_0I & \cdots \\ \\
                                        & & \cdots & &
  \end{bmatrix} \qquad \text{for}  \qquad I\in \mc{B} (\mc{L}_0^j).
\]
The matrix associated to $T(\tilde{b})$ has exactly the same scalar form, replacing $I$ by $1\in \mc{B}(\mb{C})$. Then, $T(\tilde{B})$ is invertible if and only if $T(\tilde{b})$ is invertible, and 
\[
T(\tilde{B})^{-1}  \sim \begin{bmatrix}  b_0^{(-1)} I & 0 & 0 &\cdots \\ \\
                                       b_1^{(-1)} I & b_0^{(-1)} I & 0 & \cdots \\ \\
                                        b_2^{(-1)} I & b_1^{(-1)} I & b_0^{(-1)}I & \cdots \\ \\
                                        & & \cdots & &
  \end{bmatrix} \qquad \text{for} \qquad  \tilde{b}(z)^{-1}=\sum_{k=0}^\infty b_k^{(-1)} z^k.
\]
Hence
\[
     \|B^{-1}\|=\|T(\tilde{B})^{-1}\|=\mf{M}(\tilde{b}^{-1})=\mf{m}(\tilde{b})^{-1}.
\]
\end{proof}

\begin{cor} \label{pert}
 Let $A=B+C$ for $B$ as in \eqref{opeB}. If $\|C\|<\mf{m}(\tilde{b})$, then $A$ is invertible. Moreover
\begin{equation} \label{gentrick} 
 \|A\|\leq \mf{M}(\tilde{b})+\|C\| \qquad \text{and} \qquad  \|A^{-1}\|\leq \frac{1}{\mf{m}(\tilde{b})-\|C\|}.
\end{equation}
\end{cor}
\begin{proof}
Since $B$ is invertible,  write $A=(I+CB^{-1})B$. If additionally $\|CB^{-1}\|<1$, then
\[
    \| (I+CB^{-1})^{-1} \| \leq \frac{1}{1-\|C\| \|B^{-1}\|}.
\]
\end{proof}

\begin{rem} \label{diri}
It is possible to characterise the change of coordinates $A$ in terms of Dirichlet series, and recover some of the results here and below directly from this characterisation. See for example the insightful paper \cite{HLS1997} and the complete list of references provided in the addendum \cite{HLS1999}. However, the full technology of Dirichlet series is not needed in the present context. A further development in this direction will be reported elsewhere.
\end{rem}


\section{Invertibility and bounds on the Riesz constant} \label{criteria}

A proof of \eqref{trick1} can be achieved by applying Corollary~\ref{pert} assuming that \[B=a_1M_1=a_1I.\]
Our next goal is to formulate concrete sufficient condition for the invertibility of $A$ and corresponding bounds on $r(\mc{S})$, which improve upon \eqref{trick1} whenever $r=2$. For this purpose we apply Corollary~\ref{pert} assuming that $B$ has now the three-term expansion
\[B=a_1M_1+a_3 M_3+a_9 M_9.\]

Let 
\[
\sf{T}=\{ \beta<1,\, \beta-\alpha+1>0,\, \beta+\alpha+1>0  \}. 
\]
Let
\begin{align*}
   \sf{R}_1&=\left\{\left|\alpha(\beta+1)\right|<|4\beta|\right\}\cap\{\beta>0\}\\
\sf{R}_3&=\left\{\left|\alpha(\beta+1)\right|< |4\beta|\right\} \cap\{\beta<0\}\\
\sf{R}_2&=\left\{\left|\alpha(\beta+1)\right|\geq |4\beta|\right\}=\mb{R}^2\setminus \left(\sf{R}_1\cup \sf{R}_3\right).
\end{align*}
See Figure~\ref{region}.

\begin{figure}[t]
\centerline{
\includegraphics[height=10cm, angle=0]{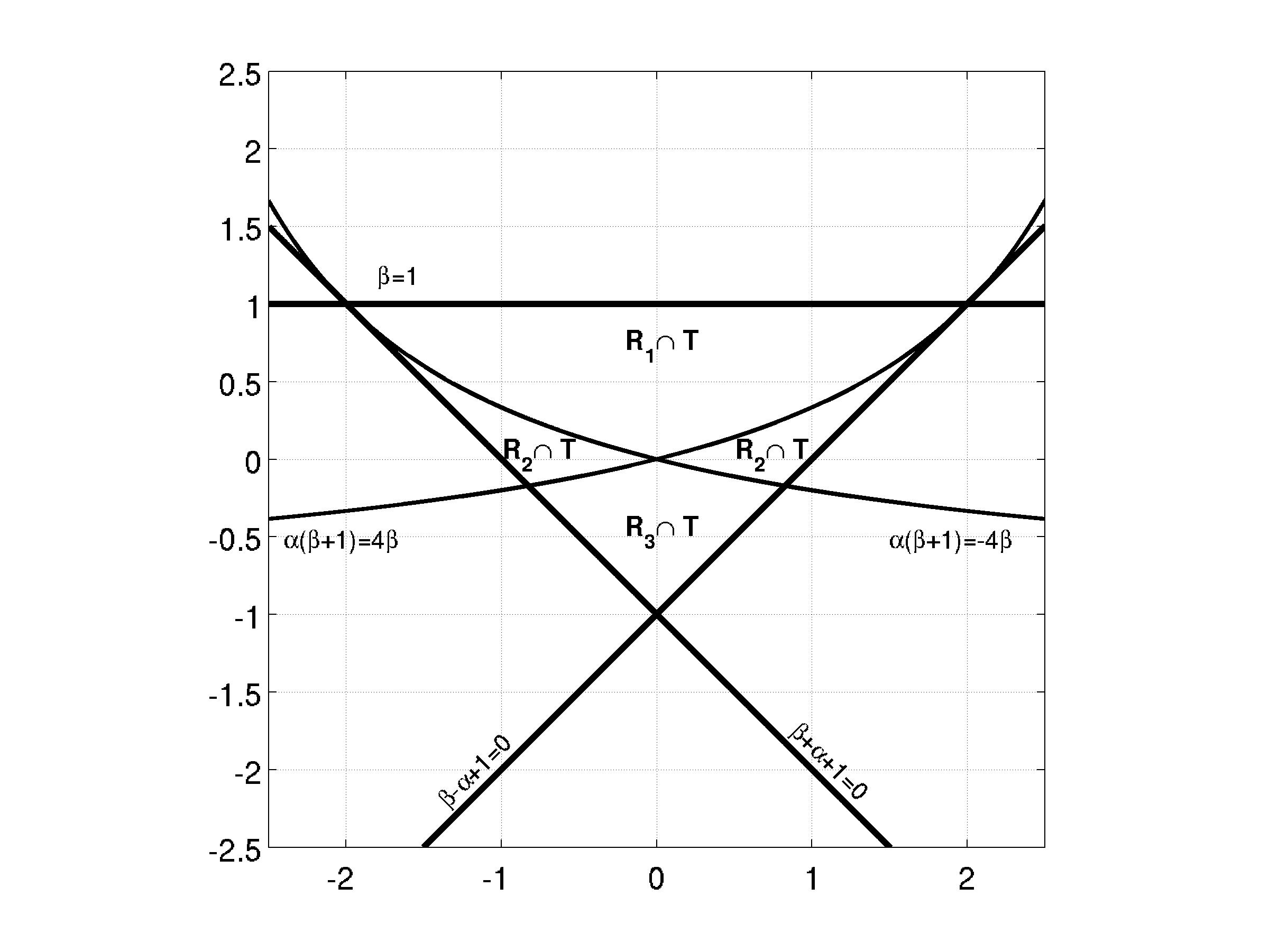}}
\caption{\label{region}Optimal region of invertibility in Lemma~\ref{toep}. In this picture the horizontal axis is $\alpha$ and the vertical axis is $\beta$.}
\end{figure}

\begin{lem} \label{toep}
Let $r=2$. Let $\alpha,\beta\in \mb{R}$. The operator $B=I+\alpha M_3+\beta M_9$ is invertible if and only if $(\alpha,\beta)\in \sf{T}$. Moreover
\begin{equation*}
 \begin{bmatrix}
\|B\| \\ \|B^{-1}\|^{-1} \end{bmatrix}=\left\{\begin{array}{ll}
   \begin{bmatrix}{1+\beta+|\alpha|}\\ {(1-\beta)\sqrt{1-\frac{\alpha^2}{4\beta}}} \end{bmatrix} 
   & \quad  (\alpha,\beta)\in\sf{R}_1\cap \sf{T} \\ \\
 \begin{bmatrix}{1+\beta+|\alpha|} \\  {1+\beta-|\alpha|} \end{bmatrix} 
 & \quad (\alpha,\beta)\in\sf{R}_2\cap \sf{T} \\ \\
\begin{bmatrix} {(1-\beta)\sqrt{\frac{\alpha^2}{4\beta}-1}} \\ {1+\beta-|\alpha|} \end{bmatrix} 
& \quad (\alpha,\beta)\in\sf{R}_3 \cap \sf{T}
\end{array} \right. 
\end{equation*}
\end{lem}
\begin{proof}
Let $\tilde{b}(z)=1+\alpha z + \beta z^2$ be associated with $B$ as in
Section~\ref{toep_s}.

The first assertion is a consequence of the following observation. If
$\alpha^2-4\beta< 0$, then $\tilde{b}(z)$ has roots $z_\pm$ conjugate
with each other and $|z_\pm|\leq 1$ if and only if $\beta\geq 1$.
Otherwise $\tilde{b}(z)$ has two real roots. If $\alpha^2-4\beta\geq
0$ and $\alpha\geq 0$, then the smallest in modulus root of
$\tilde{b}(z)$ would lie in $\overline{\mb{D}}$ if and only if
$\beta-\alpha+1\leq 0$. If $\alpha^2-4\beta\geq 0$ and $\alpha<0$,
then the root of $\tilde{b}(z)$ that is smallest in modulus would lie
in $\overline{\mb{D}}$ if and only if $\beta+\alpha+1\leq 0$.

For the second assertion, let $(\alpha,\beta)\in \sf{T}$ and $b(\theta)=|\tilde{b}(e^{i\theta})|^2$. By virtue of the Maximum Principle on $\tilde{b}(z)$ and $\frac{1}{\tilde{b}(z)}$,
\[
      \mf{M}(\tilde{b})^2=\max_{-\pi \leq \theta <\pi} b(\theta) \qquad \text{and} \qquad  \mf{m}(\tilde{b})^2=\min_{-\pi \leq \theta <\pi} b(\theta) .
\]  
Since
\begin{align*}
    b(\theta)  &=(1+\alpha \cos(\theta)+\beta \cos(2\theta))^2+(\alpha \sin(\theta)+\beta \sin(2\theta))^2 \\
    &=1+\alpha^2+\beta^2+2(\beta+1)\alpha \cos(\theta)+2 \beta \cos(2\theta),
\end{align*}
then
$b'(\theta)=0$ if and only if $(\alpha(\beta+1)+4\beta \cos(\theta))\sin(\theta)=0$. For $\sin(\theta_0)=0$, we get
$b(\theta_0)=(1+\beta+\alpha)^2$ and $b(\theta_0)=(1+\beta-\alpha)^2$. For $\cos(\theta_0)=-\frac{\alpha(\beta+1)}{4\beta}$, we get
$b(\theta_0)=(1-\frac{\alpha^2}{4\beta})(\beta-1)^2$ with the condition $\left|\frac{\alpha(\beta+1)}{4\beta}\right|\leq 1$. 
By virtue of Theorem~\ref{generic_toep}, 
we obtain the claimed statement. \end{proof}

Since $\sin_{p,q}(x)>0$ for all $x\in (0,\pi_{p,q})$, then $a_1>0$. Below we substitute $\alpha= \frac{a_3}{a_1}$ and $\beta=\frac{a_9}{a_1}$, then  apply Lemma~\ref{toep} appropriately in order to determine the invertibility of $A$ whenever pairs $(p,q)$  lie in different regions of the $(p,q)$-plane. For this purpose we establish the following hierarchy between $a_1$ and $a_j$ for $j=3,9$, whenever the latter are non-negative.

\begin{lem} \label{ajint}
For $j=3$ or $j=9$, we have $a_j<a_1$.
\end{lem}
\begin{proof}
Firstly observe that $\sin_{p,q}(\pi_{p,q}x)$ is continuous, it increases for all $x\in(0,\frac12)$ and it vanishes at $x=0$.  

Let $j=3$. Set
\[
\begin{gathered}
        I_0=\int_0^{\frac14} \sin_{p,q}(\pi_{p,q}x) [\sin(\pi x)-\sin(3\pi x)] \mr{d}x \qquad \text{and} \\
        I_1=\int_{\frac14}^{\frac12} \sin_{p,q}(\pi_{p,q}x) [\sin(\pi x)-\sin(3\pi x)] \mr{d}x .
\end{gathered}
\]
Since
\[
      \sin(\pi x)-\sin(3\pi x)= -2\sin(\pi x) \cos(2\pi x),
\]
then $I_0<0$ and $I_1>0$. As $\cos(2\pi x)$ is odd with respect to $\frac14$ and $\sin(\pi x)$ is increasing in the segment $(0,\frac12)$, then also $|I_0|<|I_1|$.
Hence  \[a_1-a_3=2\sqrt{2}(I_0+I_1)>0,\] ensuring the first statement of the lemma.

Let $j=9$. A straightforward calculation shows that $\sin(\pi x)=\sin(9\pi x)$ if and only if, 
either $\sin(4\pi x)=0$ or $\cos(4\pi x) \cos(\pi x)= \sin(4\pi x) \sin (\pi x)$.   
Thus, $\sin(\pi x)-\sin(9\pi x)$ has exactly five zeros in the segment $[0,\frac12]$ located at:
\[
     x_0=0,\,x_1=\frac{1}{10},\, x_4=\frac{1}{4},\, x_5=\frac{3}{10} \text{ and } x_8=\frac12.
\]

Set
\[
       x_2=\frac19,\, x_3=\frac{19}{90},\,x_6= \frac{13}{36}\text{ and }x_7= \frac{37}{90},
\]
and
\[
     I_k=\int_{x_k}^{x_{k+1}} \sin_{p,q}(\pi_{p,q}x) [\sin(\pi x)-\sin(9\pi x)] \mr{d}x .
\]
Then $I_k<0$ for $k=0,4$ and $I_k>0$ for $k=1,2,3,5,6,7$. Since 
\[
       \sin(9\pi x) - \sin(\pi x)< \sin\left(\pi \left(x+\frac19\right)\right) - \sin\left(9\pi\left( x+\frac19\right)\right)
\]
for all $x\in (0,\frac12)$, then
\[
      |I_0|<|I_2| \quad \text{and} \quad |I_4|<|I_6|.
\]
Hence
\[
      a_1-a_9=2\sqrt{2}\sum_{k=0}^7I_k >2\sqrt{2}(I_1+I_3+I_5+I_7)>0.
\]
\end{proof}

The next two corollaries are consequences of Corollary~\ref{pert} and Lemma~\ref{toep}, and are among the main results of this paper.

\begin{cor} \label{main_1}
\begin{equation} \label{trick2}
      \left.\begin{aligned}
    & \left(\frac{a_3}{a_1},\frac{a_9}{a_1}   \right)\in \sf{R}_2\cap \sf{T} \\ &\sum_{j\not\in\{1,9\}}^\infty |a_j| <a_1+a_9 
      \end{aligned}   \right\}\quad \Rightarrow \quad \left\{\begin{aligned} &A,A^{-1}\in \mc{B}(L^2) \\
      & r(\mc{S}) \leq \frac{\sum_{j=1}^\infty |a_j|}{a_1+a_9-\sum_{j\not\in\{1,9\}}^\infty |a_j|}. \end{aligned} \right.
\end{equation}
\end{cor}
\begin{proof}
 Let $A=B+C$ where 
\[
B=a_1I+a_3M_3+a_9M_9 \qquad \text{and} \qquad C= \sum_{j\not\in\{1,3,9\}}^\infty a_jM_j.
\]
The top on left side of \eqref{trick2} and the fact that $a_1>0$ imply
\[
        \|B^{-1}\|^{-1}=a_1-|a_3|+a_9.
\]
Thus, the bottom on the left side of \eqref{trick2} yields
\[
     \|C\|\leq \sum_{j\not\in\{1,3,9\}}^\infty |a_j|<\|B^{-1}\|^{-1},
\]
so indeed $A$ is invertible. The estimate on the Riesz constant is deduced from the triangle inequality.
\end{proof}

Since $a_1>0$, \eqref{trick2} supersedes \eqref{trick1}, only when
the pair $(p,q)$ is such that $a_9>0$. From this corollary we see below that the change of coordinates is invertible in a neighbourhood of the threshold set by the condition \eqref{break2}.  See Proposition~\ref{beyond2} and Figures~\ref{impro_fig_p=q} and \ref{th10ab}. 

\begin{cor} \label{main_2}
\begin{equation} \label{trick3}
 \begin{aligned}
    &  \left.\begin{aligned}
    & \left(\frac{a_3}{a_1},\frac{a_9}{a_1}   \right)\in \sf{R}_1\cap \sf{T} \\ &\sum_{j\not\in\{1,3,9\}}^\infty |a_j| < 
      (a_1-a_9)\left(1-\frac{a_3^2}{4a_1a_9}\right)^{\frac12} \end{aligned}   \right\} \Rightarrow  \\ 
      &\hspace{2cm} \left\{\begin{aligned} &A,A^{-1}\in \mc{B}(L^2) \\
      & r(\mc{S})\leq \frac{\sum_{j=1}^\infty |a_j|}{
      (a_1-a_9)\left(1-\frac{a_3^2}{4a_1a_9}\right)^{\frac12} - \sum_{j\not\in\{1,3,9\}}^\infty |a_j|  }. \end{aligned} \right.
\end{aligned}
\end{equation}
\end{cor}
\begin{proof}
 The proof is similar to that of Corollary~\ref{main_1}.
\end{proof}
 
We see below that Corollary~\ref{main_1} is slightly more useful than Corollary~\ref{main_2} in the context of
the dilated $p,q$-sine functions. However the latter is needed in the proof of the main Theorem~\ref{inprovement}.

It is of course natural to ask what consequences can be derived from the other statement in Lemma~\ref{toep}. For 
\[
 \left(\frac{a_3}{a_1},\frac{a_9}{a_1}   \right)\in \sf{R}_3\cap \sf{T}, 
\]
we have $\|B^{-1}\|^{-1} =a_1-|a_3|-|a_9|$. Hence the same argument as in the proofs of corollaries \ref{main_1} and \ref{main_2} would 
reduce to \eqref{trick1}, and in this case there is no improvement. 

\section{Riesz basis properties beyond the applicability of \eqref{trick1}}      \label{ribap}
Our first goal in this section is to establish that the change of coordinates map associated to the family $\mc{S}$ is invertible beyond the region of applicability of \eqref{trick1}.  We begin by recalling a calculation which was performed in the proof of
\cite[Proposition~4.1]{EdmundsGurkaLang2012} and which will be invoked several times below. Let $a(t)$ be the inverse function of $\sin'_{p,q}(\pi_{p,q}t)$. Then
\begin{equation} \label{coefaj}
  a_j(p,q)=-\frac{2 \sqrt{2} \pi_{p,q}}{j^2\pi^2} \int_0^1 \sin \left( \frac{j \pi}{\pi_{p,q}} a(t)\right) \mr{d}t   .
\end{equation}
Indeed, integrating by parts twice and changing the variable of integration to
\[
     t= \sin'_{p,q}(\pi_{p,q} x)
\]
yields
\begin{equation*}
\begin{aligned}
a_j(p,q)&=\sqrt{2}\int_0^1 \sin_{p,q}(\pi_{p,q} x) \sin(j\pi x) \mr{d}x \\
   &= 2 \sqrt{2}  \int_0^{1/2} \sin_{p,q}(\pi_{p,q} x) \sin(j\pi x) \mr{d}x \\
   & = \frac{2 \sqrt{2} \pi_{p,q}}{j\pi} \int_0^{1/2} \sin_{p,q}'(\pi_{p,q} x) \cos(j\pi x) \mr{d}x \\
   & = -\frac{2 \sqrt{2} \pi_{p,q}}{j^2\pi^2} \int_0^{1/2} [\sin_{p,q}'(\pi_{p,q} x)]' \sin(j\pi x) \mr{d}x \\
   & = -\frac{2 \sqrt{2} \pi_{p,q}}{j^2\pi^2} \int_0^1 \sin \left( \frac{j \pi}{\pi_{p,q}} a(t)\right) \mr{d}t   .
\end{aligned} 
\end{equation*}

\begin{thm}   \label{inprovement}
Let $r=2$. Suppose that the pair $(\tilde{p},\tilde{q})$ is such that the following two conditions are satisfied
\begin{enumerate}
\item \label{improa} $a_3(\tilde{p},\tilde{q}),\,a_9(\tilde{p},\tilde{q})>0$
\item \label{improc} $\sum_{j=3}^\infty |a_j(\tilde{p},\tilde{q})|=a_1(\tilde{p},\tilde{q})$.
\end{enumerate}
Then there exists a neighbourhood $(\tilde{p},\tilde{q})\in\mathcal{N}\subset (1,\infty)^2$, such that the change of coordinates $A$ is invertible for all
$(p,q)\in \mathcal{N}$. 
\end{thm}
\begin{proof}
From the Dominated Convergence Theorem, it follows that each $a_j(p,q)$ is a continuous function of the parameters $p$ and $q$. Therefore, by virtue of \eqref{coefaj} and a further application of the Dominated Convergence Theorem, also  $\sum_{j\in \mc{F}} |a_j|$ is continuous in the parameters $p$ and $q$. Here $\mc{F}$ can be any fixed set of indices, but below in this proof we only need to consider $\mc{F}=\mb{N}\setminus\{1,9\}$ for the first possibility and 
$\mc{F}=\mb{N}\setminus\{1,3,9\}$ for the second possibility.

Write $\tilde{a}_j=a_j(\tilde{p},\tilde{q})$. The hypothesis implies $\left(\frac{\tilde{a}_3}{\tilde{a}_1},\frac{\tilde{a}_9}{\tilde{a}_1}\right)\in \sf{T}$, because
\[
     0<\frac{\tilde{a}_3}{\tilde{a}_1}+\frac{\tilde{a}_9}{\tilde{a}_1}<1.
\]
Therefore 
\begin{equation}\label{ne1}
\left(\frac{a_3}{a_1},\frac{a_9}{a_1}\right)\in \mathsf{T}\cap(0,1)^2 \qquad
\forall (p,q)\in \mc{N}_1
\end{equation}
for a suitable neighbourhood $(\tilde{p},\tilde{q})\in\mc{N}_1\subset (1,\infty)^2$. Two possibilities are now in place.

\subsubsection*{First possibility} $(\frac{\tilde{a}_3}{\tilde{a}_1},\frac{\tilde{a}_9}{\tilde{a}_1})\in \sf{R}_2\cap \sf{T}$. Note that $\sum_{j\not\in\{1,9\} } |\tilde{a}_j|<\tilde{a}_1+\tilde{a}_9$ is an immediate consequence of \ref{improa} and \ref{improc}. By continuity of all quantities involved, there exists a neighbourhood $(\tilde{p},\tilde{q})\in\mathcal{N}_2\subset (1,\infty)^2$ such that the left hand side and hence the right hand side of \eqref{trick2} hold true for all $(p,q)\in \mathcal{N}_2$.

\subsubsection*{Second possibility} $(\frac{\tilde{a}_3}{\tilde{a}_1},\frac{\tilde{a}_9}{\tilde{a}_1})\in \sf{R}_1\cap \sf{T}$. 
Substitute $\alpha=\frac{\tilde{a}_3}{\tilde{a}_1}$ and $\beta=\frac{\tilde{a}_9}{\tilde{a}_1}$. If $(\alpha,\beta)\in\sf{R}_1\cap (0,1)^2$, then
\begin{equation}  
  \label{in_trick3}
        1-\beta-\alpha<(1-\beta)\sqrt{1-\frac{\alpha^2}{4\beta}}.
\end{equation}
Indeed, the conditions on $\alpha$ and $\beta$ give
\[
    0<\alpha,\beta <1, \quad \alpha(\beta+1)<4\beta \quad \text{and} \quad \alpha+\beta<1.
\]
As $\beta>\frac{\alpha}{4-\alpha}$,
\[
      \sqrt{1-\frac{\alpha^2}{4\beta}}>\sqrt{\frac{4-4\alpha+\alpha^2}{4}}=1-\frac{\alpha}{2}.
\]
Thus
\[
   (1-\beta)\sqrt{1-\frac{\alpha^2}{4\beta}}>(1-\beta)\left(1-\frac{\alpha}{2}\right)=
    1-\beta-\frac{\alpha}{2}+\frac{\alpha \beta}{2}>1-\beta-\alpha
\]
which is \eqref{in_trick3}.
Hence
 \[
      \sum_{j\not\in\{1,3,9\}}^\infty |\tilde{a}_j| = (\tilde{a}_1-\tilde{a}_9-\tilde{a}_3)<(\tilde{a}_1-\tilde{a}_9)\sqrt{1-\frac{\tilde{a}_3^2}{4\tilde{a}_1\tilde{a}_9}}.
\]
Thus, once again by continuity of all quantities involved, there exists a neighbourhood $(\tilde{p},\tilde{q})\in\mathcal{N}_3\subset (1,\infty)^2$ such that the left hand side and hence the right hand side of \eqref{trick3} hold true for all $(p,q)\in \mathcal{N}_3$.

\medskip

The conclusion follows by defining either $\mc{N}=\mc{N}_1\cap\mc{N}_2$ or $\mc{N}=\mc{N}_1\cap\mc{N}_3$.
\end{proof}

We now examine other further consequences of the corollaries \ref{main_1} and \ref{main_2}.

\begin{thm}  \label{impro_implicit}
Any of the following conditions ensure the invertibility of the change of coordinates map $A:L^r\longrightarrow L^r$.
\begin{enumerate}
\item \label{aimplicit} ($r>1$): 
\begin{equation}   \label{eqaimplicit}
      \frac{\pi_{p,q}}{a_1} <\frac{2\sqrt{2} \pi^2}{\pi^2-8}.
\end{equation}
\item \label{bimplicit} ($r=2$): $a_3>0$, $a_9>0$, $a_3(a_1+a_9)\geq 4a_9a_1$ and
\[
      \frac{\pi_{p,q}}{a_1+a_9} <\frac{\pi^2}{\left(\frac{\pi^2}{8}-\frac{82}{81}\right)2\sqrt{2} }.
\]
\item \label{cimplicit} ($r=2$): $a_3>0$, $a_9>0$, $a_3(a_1+a_9)< 4a_9a_1$ and
\[
      \frac{\pi_{p,q}}{(a_1-a_9)\left( 1-\frac{a_3^2}{4a_1a_9}   \right)^{1/2}} 
      <\frac{\pi^2}{\left(\frac{\pi^2}{8}-\frac{91}{81}\right)2\sqrt{2} }.
\]
\end{enumerate}
\end{thm}
\begin{proof}
From  \eqref{coefaj}, it follows that
\begin{equation}
     \sum_{j\not\in\{1\}} |a_j|\leq \frac{2 \sqrt{2} \pi_{p,q}}{\pi^2}\left(\frac{\pi^2}{8}-1\right)  \label{sum1}.
\end{equation}
Hence the condition \ref{aimplicit} implies that the hypothesis \eqref{trick1} is satisfied.

By virtue of Lemma~\ref{ajint}, it is guaranteed that 
\[
     \left(\frac{a_3}{a_1},\frac{a_9}{a_1}\right)\in (0,1)^2\subset \sf{T}
\]
 in the settings of \ref{bimplicit} or  \ref{cimplicit}.
From  \eqref{coefaj}, it also follows that
 \begin{align}    
     \quad \sum_{j\not\in\{1,9\}} |a_j|& \leq \frac{2 \sqrt{2} \pi_{p,q}}{\pi^2}\left(\frac{\pi^2}{8}-\frac{82}{81} \right)\label{sum2}
     \quad \text{and that} \\
     \quad \sum_{j\not\in \{1,3,9\}} |a_j|&\leq \frac{2 \sqrt{2} \pi_{p,q}}{\pi^2} \left(\frac{\pi^2}{8}-\frac{91}{81}  \right)\label{sum3} .
\end{align}
Combining each one of these assertions with \eqref{trick2} and \eqref{trick3}, respectively, immediately leads to the claimed statement.
\end{proof}

We recover \cite[Corollary~4.3]{EdmundsGurkaLang2012} from the part \ref{aimplicit} of this theorem by observing that for all $p,q>1$,
\[
       a_1\geq 2\sqrt{2} \int_0^{1/2} 2x \sin(\pi x) \mr{d}x = \frac{4\sqrt{2}}{\pi^2}. 
\]
In fact, for $(p,q)\in(1,2)^2$, the better estimate
\[
       a_1\geq 2\sqrt{2} \int_0^{1}  \sin^2(\pi x) \mr{d}x = \frac{\sqrt{2}}{2}, 
\]
ensures invertibility of $A$ for all $r>1$ whenever
\begin{equation} \label{impro_casea}
       \pi_{p,q}<\frac{2\pi^2}{\pi^2-8}.
\end{equation}
See figures~\ref{th10ab} and \ref{th10c}. 


\section{The case of equal indices} \label{casep=q}

We now consider in closer detail the particular case $p=q<2$. Our analysis requires setting various sharp upper and lower bounds on the 
coefficients $a_j(p,p)$ for $j=1,3,5,7,9$. This is our first goal.

\begin{figure}
\centerline{
\includegraphics[width=9cm, angle=0]{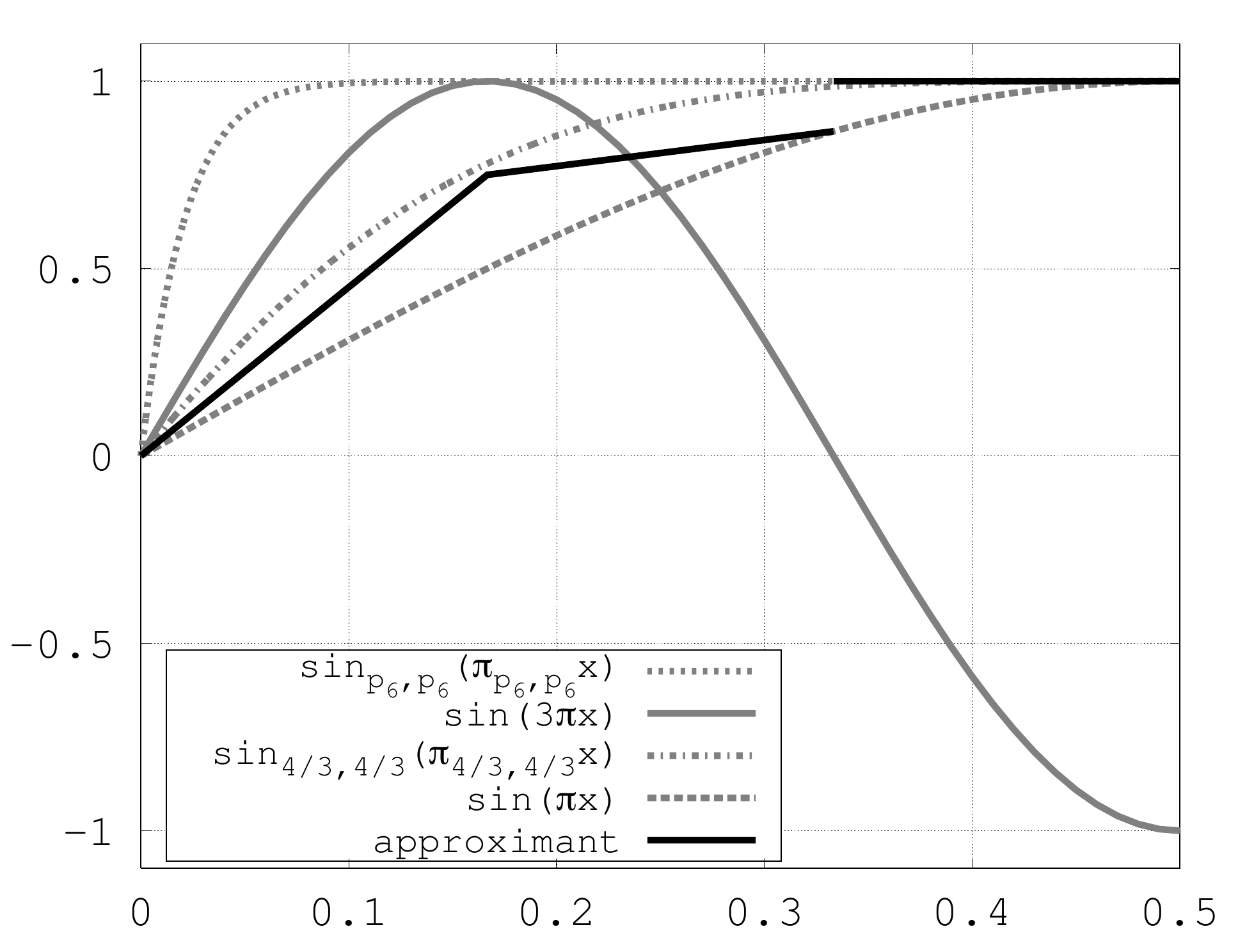}}
\caption{Approximants $\ell_j(x)$ employed to show bound \ref{a3l} in
  Lemma~\ref{ajpositive}. For reference we also show
  $\sin_{p_6,p_6}(\pi_{p_6,p_6}x)$,   $\sin(3\pi x)$, $\sin_{\frac43,\frac43}(\pi_{\frac43,\frac43}x)$ and 
  $\sin_{2,2}(\pi x)=\sin(\pi x)$.
\label{fig:interp}
}
\end{figure}

\begin{lem}   \label{ajpositive}
\
\begin{enumerate}
\item \label{a3l} $a_3(p,p)>0$ for all $1<p\leq \frac43$
\item  \label{a5l} $a_5(p,p)>0$ for all $1<p\leq \frac65$
\item  \label{a7l} $a_7(p,p)>0$ for all $1<p\leq \frac65$
\item  \label{a9l} $a_9(p,p)>0$ for all $1<p\leq \frac{12}{11}$
\end{enumerate}
\end{lem}
\begin{proof}  All the stated bounds are determined by integrating a
  suitable  approximation of $\sin_{p,p}(\pi_{p,p}x)$. Each one
  requires a different set of quadrature points, but the general
  structure of the arguments in all cases is similar.  Without further
  mention, below we repeatedly use the fact that in terms of hypergeometric functions,
\[
       \sin_{p,q}^{-1}(y)=\int_0^y \frac{\mr{d}x}{(1-x^{q})^{\frac1p}}=y \ \hypa\!\! \left( \frac1p,\frac1q;\frac1q+1;y^q\right) \qquad \forall y\in[0,1].
\]

\subsubsection*{Bound \ref{a3l}} Let
\[
 \{x_j\}_{j=0}^3=\left\{0,\frac16,\frac13,\frac12\right\} \qquad \text{and} \qquad
 \{y_j\}_{j=0}^3=\left\{0,\frac34,\frac{\sqrt{3}}{2},1\right\}.
\]
For $x\in[x_j,x_{j+1})$ let 
\begin{equation*}
      \ell_j(x)=\frac{y_{j+1}-y_j}{x_{j+1}-x_j} (x-x_j)+y_j \text{ for } j=0,1 \qquad
\text{and} \qquad
      \ell_2(x)=1,
\end{equation*}
see Figure~\ref{fig:interp}. Since
\begin{align*}
     \sin_{\frac43,\frac43}^{-1}(y_1)   &=\left(\frac34\right) \hypa\left(\frac34,\frac34;\frac74;\left(\frac34\right)^{\frac43}\right) \\
&< \frac{105}{100}<\frac{110}{100}<\frac{\pi\sqrt{2}}{4}=\frac{\pi_{\frac43,\frac43}}{6}
\end{align*}
and $\sin_{p,p}(t)$ is an increasing function of $t\in(0,\frac{\pi_{p,p}}{2})$, then
 \begin{equation*}   \label{ineq_sin43}
       \sin_{\frac43,\frac43}\left(\pi_{\frac43,\frac43}x_1\right)>y_1.
 \end{equation*}
 
According to \cite[Corollary~4.4]{BushellEdmunds2012}\footnote{See also \cite[Lemma~5]{BBCDG2006}.}, $\sin_{p,p}(\pi_{p,p}x)$ increases as $p$ decreases for any fixed $x\in (0,1)$.   Let $p$ be as in the hypothesis. Then
  \begin{equation*}   \label{ineq_sin}
       \sin_{p,p}\left(\pi_{p,p} x_1\right)>y_1
 \end{equation*}
and similarly 
 \begin{equation*} 
       \sin_{p,p}\left(\pi_{p,p}x_2\right)>\sin_{2,2}\left(\pi_{2,2}x_2\right)=y_2.
 \end{equation*}
By virtue of \cite[Lemma~3]{BBCDG2006} the function $\sin_{p,p}(t)$ is strictly concave for $t\in(0,\frac{\pi_{p,p}}{2})$. 
Then, in fact,
\begin{align*}
     \sin_{p,p}(\pi_{p,p}x)&> \ell_0(x)=\frac92 x & \forall x\in\left(x_0,x_1 \right)  \\
     \sin_{p,p}(\pi_{p,p}x)&> \ell_1(x)=\left( 3\sqrt{3}-\frac92\right)x+\frac{3-\sqrt{3}}{2} & \forall x\in\left(x_1,x_2 \right)   .
\end{align*}

Let
\[
        I_j=2\sqrt{2}\int_{x_j}^{x_{j+1}} \ell_j(x) \sin(3\pi x) \mr{d}x.
\]
Since $\sin(3\pi x)\leq 0$ for $x\in(\frac13,\frac12)$ and $|\sin_{p,p}(\pi_{p,p}x)|\leq 1$, 
\begin{equation*}
\begin{aligned}
a_3(p,p)& = 2\sqrt{2} \int_0^{\frac12}  \sin_{p,p}(\pi_{p,p}x) \sin(3\pi x) \mr{d}x  \\
    & > I_0+I_1+I_2\\
  & = 2\sqrt{2}\left(\frac{1}{2\pi^2} +\frac{(\pi-2)\sqrt{3}+3}{6\pi^2} - \frac{1}{3\pi}\right)>0.
  \end{aligned}
\end{equation*} 


\subsubsection*{Bound \ref{a5l}} Note that
\[
      \pi_{\frac65,\frac65}=\frac{10\pi}{3}.
\] 
Set
\[
 \{x_j\}_{j=0}^4=\left\{0,\frac{1}{10},\frac15,\frac25,\frac12\right\} \text{ and }
 \{y_j\}_{j=0}^4=\left\{0,\frac{171}{250},\frac{93}{100},\frac{99}{100},1\right\}.
\]
Then
\[
     \sin_{\frac65,\frac65}^{-1}(y_1)=y_1 \ \hypa\left(\frac56,\frac56;\frac{11}{6};y_1^{\frac65}\right) < 1<\frac{\pi}{3}=\pi_{\frac65,\frac65}x_1
\]
and so
 \begin{equation*}
       \sin_{\frac65,\frac65}\left(\pi_{\frac65,\frac65}x_1\right)>y_1.
 \end{equation*}
Also
\[
     \sin_{\frac65,\frac65}^{-1}(y_2)< 2<\pi_{\frac65,\frac65}x_2 \quad \text{and} \quad
     \sin_{\frac65,\frac65}^{-1}(y_3)< 3<\pi_{\frac65,\frac65}x_3,
\]
so
 \begin{equation*}
       \sin_{\frac65,\frac65}\left(\pi_{\frac65,\frac65}x_j\right)>y_j \qquad j=2,3.
 \end{equation*}
 
 Let $p$ be as in the hypothesis. Then, similarly to the previous case \ref{a3l},
  \begin{equation}  \label{fora5l}
         \sin_{p,p}\left(\pi_{p,p} x_j\right)>y_j \qquad j=1,2,3.
 \end{equation}
Set
\begin{align*}
      \ell_j(x)&=\frac{y_{j+1}-y_j}{x_{j+1}-x_j} (x-x_j)+y_j \qquad j=0,1,3 \\
      \ell_2(x)&=1.
\end{align*}
By strict concavity and \eqref{fora5l},
\[
     \sin_{p,p}(\pi_{p,p}x)> \ell_j(x) \qquad \qquad \forall x\in\left(x_j,x_{j+1} \right)\quad j=0,1,3. 
\]
Let
\[
        I_j=2\sqrt{2}\int_{x_j}^{x_{j+1}} \ell_j(x) \sin(5\pi x) \mr{d}x \qquad \qquad j=0,1,2,3.
\]
Then
\begin{equation*}
a_5(p,p)>\sum_{j=0}^3 I_j>\frac{3}{100}>0
\end{equation*} 
as claimed.


\subsubsection*{Bound \ref{a7l}} Let $p$ be as in the hypothesis. Set
\begin{align*}
 \{x_j\}_{j=0}^5=\left\{0,\frac{1}{14},\frac17,\frac27,\frac37,\frac12\right\} \text{ and }
 \{y_j\}_{j=0}^5=\left\{0,\frac{283}{500},\frac{106}{125},1,1,1\right\}.
\end{align*}
Then
\begin{align*}
     \sin_{\frac65,\frac65}^{-1}(y_1)&< \frac{73}{100}<\pi_{\frac65,\frac65}x_1 \quad \text{and} \quad
     \sin_{\frac65,\frac65}^{-1}(y_2)<\frac{147}{100}<\pi_{\frac65,\frac65}x_2.
\end{align*}
Hence
  \begin{equation*}  
         \sin_{p,p}\left(\pi_{p,p} x_j\right)>y_j \qquad j=1,2.
 \end{equation*}
Put
\begin{align*}
      \ell_j(x)&=\frac{y_{j+1}-y_j}{x_{j+1}-x_j} (x-x_j)+y_j \qquad j=0,1 \\
      \ell_4(x)&=1.
\end{align*}
Then,
\[
     \sin_{p,p}(\pi_{p,p}x)> \ell_j(x) \qquad \qquad \forall x\in\left(x_j,x_{j+1} \right)\quad j=0,1. 
\]
Let
\begin{align*}
        I_j&=2\sqrt{2}\int_{x_j}^{x_{j+1}} \ell_j(x) \sin(7\pi x) \mr{d}x & j=0,1,4 \\
        I_j&=2\sqrt{2}\int_{x_j}^{x_{j+1}} \sin_{p,p}(\pi_{p,p}x) \sin(7\pi x) \mr{d}x & j=2,3.
\end{align*}
Since $\sin(7\pi x)$ is negative for $x\in(x_2,x_3)$ and positive for $x\in(x_3,x_4)$, then  $I_2+I_3>0$.
Hence
\begin{equation*}
a_7(p,p)>I_0+I_1+I_4>\frac{3}{1000}>0.
\end{equation*}

\subsubsection*{Bound \ref{a9l}} 
Note that
\[
   \pi_{\frac{12}{11},\frac{12}{11}}=\frac{11\pi \sqrt{2}}{3(\sqrt{3}-1)}.
\]
Let $p$ be as in the hypothesis. Set
\[
 \{x_j\}_{j=0}^5=\left\{0,\frac{1}{18},\frac19,\frac13,\frac49,\frac12\right\} \text{ and }
 \{y_j\}_{j=0}^5=\left\{0,\frac{17}{24},\frac{15}{16},\frac{15}{16},\frac{15}{16},\frac{15}{16}\right\}.
\]
Then
\begin{align*}
     \sin_{\frac{12}{11},\frac{12}{11}}^{-1}(y_1)&< \frac{112}{100}<\pi_{\frac{12}{11},\frac{12}{11}}x_1 \quad \text{and} \quad
     \sin_{\frac{12}{11},\frac{12}{11}}^{-1}(y_2)<\frac{233}{100}<\pi_{\frac{12}{11},\frac{12}{11}}x_2.
\end{align*}
Hence
  \begin{equation*}  
         \sin_{p,p}\left(\pi_{p,p} x_j\right)>y_j \qquad j=1,2.
 \end{equation*}
Put
\begin{align*}
      \ell_j(x)&=\frac{y_{j+1}-y_j}{x_{j+1}-x_j} (x-x_j)+y_j \qquad j=0,1 \\
      \ell_3(x)&=1 \qquad \qquad \ell_4(x)=\frac{15}{16}.
\end{align*}
Then,
\[
     \sin_{p,p}(\pi_{p,p}x)> \ell_j(x) \qquad \qquad \forall x\in\left(x_j,x_{j+1} \right)\quad j=0,1,4. 
\]
Let
\begin{align*}
        I_j&=2\sqrt{2}\int_{x_j}^{x_{j+1}} \ell_j(x) \sin(9\pi x) \mr{d}x & j=0,1,3,4 \\
        I_2&=2\sqrt{2}\int_{x_2}^{x_{3}} \sin_{p,p}(\pi_{p,p}x) \sin(9\pi x) \mr{d}x .
\end{align*}
Then  $I_2>0$. Hence
\begin{equation*}
a_9(p,p)>I_0+I_1+I_3+I_4=2\sqrt{2}\left(\frac{23}{216\pi^2}-\frac{1}{72\pi}\right)>0.
\end{equation*} 
\end{proof}

The next statement is a direct consequence of combining \ref{a3l} and \ref{a9l} from this lemma with Theorem~\ref{inprovement}.

\begin{cor}   \label{beyonda}
Set $r=2$ and suppose that $1<\tilde{p}_2<\frac{12}{11}$ is such that
\[
\sum_{j=3}^\infty |a_j(\tilde{p}_2,\tilde{p}_2)|=a_1(\tilde{p}_2,\tilde{p}_2).
\]
There exists $\varepsilon>0$ such that  $A$ is invertible for all
$p\in (\tilde{p}_2-\varepsilon,\tilde{p}_2+\varepsilon)$. 
\end{cor}

See Figure~\ref{impro_fig_p=q}.
 
\begin{rem} \label{rem_gap}
In \cite{BBCDG2006} it was claimed that the hypothesis of \eqref{trick1} held true whenever $p=q\geq p_1$ for a suitable $1<p_1<\frac{12}{11}$. The argument supporting this claim \cite[\S4]{BBCDG2006} was separated into two cases: $p\geq 2$ and $\frac{12}{11}\leq p<2$. With our definition\footnote{The Fourier coefficients in \cite{BBCDG2006} differ from $a_j(p,p)$ by a factor of $\sqrt{2}$. Note that the ground eigenfunction of the $p$-Laplacian equation in \cite{BBCDG2006} is denoted by $S_p(x)$ and it equals $\sin_{p,p}(x)$ as defined above. A key observation here is the $p$-Pythagorean identity $|\sin_{p,p}(x)|^p+|\sin_{p,p}'(x)|^p=1=|S_p(x)|^p+|S_p'(x)|^p$.}  of the Fourier coefficients, in the latter case it was claimed that $|a_j|$ was bounded above by
\[
  \frac{2\sqrt{2}\pi_{\frac{12}{11},\frac{12}{11}}}{j^2\pi^2}\left(\int_0^{\frac12} \sin_{p,p}''(\pi_{p,p}t)^2 \mr{d}t \right)^{1/2}
  \left(\int_0^{\frac12} \sin(j \pi t)^2 \mr{d}t \right)^{1/2}.
\]
As it turns,
there is a missing power 2 in the term $\pi_{\frac{12}{11},\frac{12}{11}}$ for this claim to be true. This corresponds to taking second derivatives of $\sin_{p,p}(\pi_{p,p}t)$ and it can be seen by applying the Cauchy-Schwartz inequality in \eqref{coefaj}. 
The missing factor is crucial in the argument and renders the proof of \cite[Theorem~1]{BBCDG2006} incomplete in the latter case.
\end{rem}

In the paper \cite{BushellEdmunds2012} published a few years later, it was claimed that the hypothesis of \eqref{trick1} held true
for $p=q\geq \tilde{p}_1$ where $\tilde{p}_1$ is defined by
\eqref{bueva}. It was then claimed that an approximated solution of
\eqref{bueva} was near $1.05 <\frac{12}{11}$. An accurate numerical
approximation of \eqref{bueva}, based on analytical bounds on
$a_1(p,p)$, give the correct digits $\tilde{p}_1\approx
1.198236>\frac{12}{11}$. Therefore neither the results of  
\cite{BBCDG2006} nor those of \cite{BushellEdmunds2012} include a complete proof of invertibility of the change of coordinates 
in a neighbourhood of $p=\frac{12}{11}$.

Accurate numerical estimation of $a_1(p,p)$ show that the identity
\eqref{eqaimplicit}  is valid as long as $p>\hat{p}_1\approx
1.158739>\frac{12}{11}$, which improves slightly upon the value $\tilde{p}_1$ from
\cite{BushellEdmunds2012}. However, as remarked in \cite{BushellEdmunds2012}, 
the upper bound 
\[
     |a_j|\leq \frac{2\sqrt{2} \pi_{p,p}}{j^2 \pi^2}
\]
ensuring \eqref{sum1} and hence the validity of
Theorem~\ref{impro_implicit}-\ref{aimplicit}, is too crude for small
values of $p$. Note for example that the correct regime is
$a_j(p,p)\to \frac{2\sqrt{2}}{j\pi}$  whereas $\pi_{p,p}\to \infty$ as
$p\to 1$ (see Appendix~\ref{ap1}). Therefore, in order to determine
invertibility of $A$ in the vicinity of $p=q=\frac{12}{11}$, 
it is necessary to find sharper bounds for the first few terms $|a_j|$, and employ \eqref{trick1} directly.
This is the purpose of the next lemma. See Figure~\ref{impro_fig_p=q}.

\begin{lem} \label{ajbounds}
Let $1<p\leq \frac65$. Then
\begin{enumerate}
\item \label{a1l} $a_1(p,p)>\frac{839}{1000}$
\item  \label{a3u} $a_3(p,p)<\frac{151}{500}$
\item  \label{a5u} $a_5(p,p)<\frac{181}{1000}$
\item  \label{a7u} $a_7(p,p)<\frac{13}{100}$
\end{enumerate}
\end{lem}
\begin{proof} We proceed in a similar way as in the proof of Lemma~\ref{ajpositive}. Let $p$ be as in the hypothesis. 

\subsubsection*{Bound \ref{a1l}}  Set
\[
 \{x_j\}_{j=0}^3=\left\{0,\frac{31}{250}, \frac{101}{500},\frac12\right\} \quad \text{and} \quad
 \{y_j\}_{j=0}^5=\left\{0,\frac45, \frac{19}{20}, 1\right\}.
\]
Then
\begin{align*}
     \sin_{\frac{6}{5},\frac{6}{5}}^{-1}(y_1)&< \frac{129}{100}<\pi_{\frac{6}{5},\frac{6}{5}}x_1 \quad \text{and} \quad
     \sin_{\frac{6}{5},\frac{6}{6}}^{-1}(y_2)<\frac{211}{100}<\pi_{\frac{6}{5},\frac{6}{5}}x_2
\end{align*}
and so
  \begin{equation*}  
         \sin_{p,p}\left(\pi_{p,p} x_j\right)>y_j \qquad j=1,2.
 \end{equation*}
Let
\[
\left.\begin{aligned}
      \ell_j(x)&=\frac{y_{j+1}-y_j}{x_{j+1}-x_j} (x-x_j)+y_j  \\
      I_j&=2\sqrt{2}\int_{x_j}^{x_{j+1}} \ell_j(x) \sin(\pi x) \mr{d}x 
\end{aligned}\right\}
\qquad j=0,1,2.
\]
Then,
\[
     \sin_{p,p}(\pi_{p,p}x)> \ell_j(x) \qquad \qquad \forall x\in\left(x_j,x_{j+1} \right)\quad j=0,1,2. 
\]
Hence
\begin{equation*}
a_1(p,p)>I_0+I_1+I_2>\frac{839}{1000}.
\end{equation*} 

\subsubsection*{Bound \ref{a3u}}  Set
\[
 \{x_j\}_{j=0}^2=\left\{0,\frac13, \frac12\right\} \qquad \text{and} \qquad
 \{y_j\}_{j=0}^2=\left\{0,\frac{99}{100},  1\right\}.
\]
Then
\begin{equation*}
     \sin_{\frac{6}{5},\frac{6}{5}}^{-1}(y_1)< 3 <\pi_{\frac{6}{5},\frac{6}{5}}x_1 
\qquad \text{ and so } \qquad
         \sin_{p,p}\left(\pi_{p,p} x_1\right)>y_1.
 \end{equation*}
Let
\[
\begin{aligned}
      \ell_0(x)&=1 \\
      \ell_1(x)&=\frac{y_{2}-y_1}{x_{2}-x_1} (x-x_1)+y_1  \\
      I_j&=2\sqrt{2}\int_{x_j}^{x_{j+1}} \ell_j(x) \sin(3\pi x) \mr{d}x  \qquad j=0,1.
\end{aligned}
\]
Then,
\[
     \sin_{p,p}(\pi_{p,p}x)> \ell_1(x) \qquad \qquad \forall x\in\left(x_1,x_{2} \right)
\]
and hence
\begin{equation*}
a_3(p,p)<I_0+I_1<\frac{151}{500}.
\end{equation*} 

\subsubsection*{Bound \ref{a5u}}  Set
\[
 \{x_j\}_{j=0}^2=\left\{0,\frac15, \frac25, \frac12\right\}
\]
and let 
\[
\begin{aligned}
      I_j&=2\sqrt{2}\int_{x_j}^{x_{j+1}} \sin_{p,p}(\pi_{p,p}x) \sin(5\pi x) \mr{d}x  \qquad j=0,1\\
      I_2&=2\sqrt{2}\int_{x_2}^{x_{3}} \sin(5\pi x) \mr{d}x. 
\end{aligned}
\]
Then, $I_0+I_1<0$, so
\begin{equation*}
a_5(p,p)<I_2=\frac{2\sqrt{2}}{5\pi}<\frac{181}{1000}.
\end{equation*} 

\subsubsection*{Bound \ref{a7u}}  Set
\[
 \{x_j\}_{j=0}^4=\left\{0,\frac17, \frac27,\frac{5}{14},\frac37, \frac12\right\}
\]
and
\[
\begin{aligned}
  I_j&=2\sqrt{2}\int_{x_j}^{x_{j+1}} \sin_{p,p}(\pi_{p,p}x) \sin(7\pi
  x) \mr{d}x  \qquad j=0,1,3,4.\\
  I_2&=2\sqrt{2}\int_{x_2}^{x_{3}} \sin(7\pi x) \mr{d}x
\end{aligned}
\]
Then, $I_0+I_1<0$ and $I_3+I_4<0$, so
\begin{equation*}
a_7(p,p)<I_2=\frac{2\sqrt{2}}{7\pi}<\frac{13}{100}.
\end{equation*} 
\end{proof}

The following result fixes the proof of the claim made in \cite[\S4 Claim~2]{BBCDG2006} and improves the threshold of
invertibility determined in \cite[Theorem~4.5]{BushellEdmunds2012}.

\begin{thm} \label{fixingBBCDG}
There exists $1<p_3<\frac65$, such that
\begin{equation} \label{proofto1211}
     \pi_{p,p}< \frac{[a_1(p,p)-a_3(p,p)-a_5(p,p)-a_7(p,p)]\pi^2}{2\sqrt{2}\left( \frac{\pi^2}{8}-1- \frac19-\frac{1}{25}-\frac{1}{49}  \right)} \qquad  \forall  p\in\left(p_3, \frac{6}{5}\right).
\end{equation}
The family $\mc{S}$ is a Schauder basis of $L^r(0,1)$ for all $p_3<p=q<\frac{6}{5}$ and $r>1$.
\end{thm}
\begin{proof}
Both sides of \eqref{proofto1211} are continuous functions of the parameter $p>1$. 
The right hand side is bounded. The left side is decreasing as $p$ increases and $\pi_{p,p}\to \infty$ as $p\to 1$. By virtue of Lemma~\ref{ajbounds},
\[
    \pi_{\frac65,\frac65}= \frac{10\pi}{3}<12<\frac{\left(a_1(\frac65,\frac65)-a_3(\frac65,\frac65)-a_5(\frac65,\frac65)-a_7(\frac65,\frac65)\right) \pi^2}{2\sqrt{2}\left( \frac{\pi^2}{8}-1-\frac19-\frac{1}{25}-\frac{1}{49}  \right)}.
\]
Hence the first statement is ensured as a consequence of the Intermediate Value Theorem.

From \eqref{coefaj}, it follows that
\[
     \sum_{j\not\in\{1,3,5,7\}} |a_j(p,p)|< \frac{2\sqrt{2}\pi_{p,p}}{\pi^2}\left( \frac{\pi^2}{8}-1-\frac19-\frac{1}{25}-\frac{1}{49}  \right)
\]
for all $p_3<p<\frac{6}{5}$. Lemma~\ref{ajpositive} guarantees positivity of $a_j$ for $j=3,5,7$. Then, by re-arranging this inequality, the second statement becomes a direct consequence of  \eqref{trick1}. 
\end{proof}

A sharp numerical approximation of the solution of the equation with equality in \eqref{proofto1211} gives $p_3\approx1.087063<\frac{12}{11}$. See Figure~\ref{impro_fig_p=q}.

\section{The thresholds for invertibility and the regions of improvement} \label{nume}

If sharp bounds on the first few Fourier coefficients $a_j(p,q)$ are at hand, the approach employed above for the proof of Theorem~\ref{fixingBBCDG} can also be combined with the criteria \eqref{trick2} or \eqref{trick3}.  A natural question is whether
this would lead to a positive answer to the question of invertibility for $A$, whenever
\[
      \sum_{k=3}^\infty a_j \geq a_1.
\]
In the case of \eqref{trick2}, we see below that this is indeed the case. The key statement is summarised as follows.

\begin{pro} \label{beyond2}
  Let $r=2$ and $5\leq k\not \equiv_20$. Suppose that
\begin{enumerate}
 \item $a_3>0$, $a_9>0$ and $a_j\geq 0$ for all other $5\leq j \leq k$.
\item $a_3(a_1+a_9)> 4a_9a_1$.
\end{enumerate}
If
\begin{equation} \label{verygeneral}
     \pi_{p,q}<\left(a_1+a_9-\sum_{\substack{3\leq j\leq k \\ j\not\in \{1,9\}}} a_j\right) \frac{\pi^2}{2\sqrt{2}\left(\frac{\pi^2}{8}-\sum_{\substack{1\leq j\leq k \\ j\not\equiv_2 0}}^k \frac{1}{j^2} \right)},
\end{equation}
then $A$ is invertible.
\end{pro}
\begin{proof}
Assume that the hypotheses are satisfied. The combination of \eqref{coefaj} and \eqref{verygeneral} gives
\[
     \sum_{j=k+1}^{\infty} |a_j|\leq \frac{2\sqrt{2}\pi_{p,q}}{\pi^2}\left( \frac{\pi^2}{8}- \sum_{\substack{1\leq j\leq k \\ j\not\equiv_2 0}}^k \frac{1}{j^2}  \right)<a_1+a_9-\sum_{\substack{3\leq j\leq k \\ j\not\in \{1,9\}}}a_j.
\]
Then 
\[
       \sum_{j\not\in\{1,9\}}|a_j|=\sum_{\substack{3\leq j\leq k \\j\not\in\{1,9\} }}a_j+\sum_{\substack{ j>k \\ j\not\in\{1,9\}}}|a_j|<a_1+a_9
\] 
and so the conclusion follows from \eqref{trick2}. 
\end{proof}

We now discuss the connection between the different statements
established in the previous sections with those of the papers
\cite{BBCDG2006}, \cite{BushellEdmunds2012} and
\cite{EdmundsGurkaLang2012}. 
For this purpose we consider various accurate
approximations of $a_j$ and $\sum a_j$. These approximations are based on
the next explicit formulae:
\[
      \pi_{p,q}=\frac{2 \on{B}\left(\frac1q,\frac{p-1}{p}\right)}{q}=\frac{2\ \Gamma\left(\frac{p-1}{p}\right)
                  \Gamma\left(\frac1q\right)}{q\ \Gamma\left(\frac{p-1}{p}+\frac1q\right)}
\]
and
\[
    \begin{aligned} 
     a_j(p,q)&=\frac{2\sqrt{2}}{j\pi} \int_0^1 \cos\left(\frac{j\pi x}{\pi_{p,q}}
                  \ \hypa\! \! \left(\frac1p,\frac1q;1+\frac1q;x^q\right)\right) \mr{d}x \\
                  &=\frac{2\sqrt{2}}{j\pi}\int_0^1 \cos\left( \frac{j\pi}{2} \ \mc{I}\! \!\left(\frac1q,\frac{p-1}{p};x^q\right)   \right) \mr{d}x.
                  \end{aligned}
\]
Here $\mc{I}$ is the incomplete beta function,
$\on{B}$ is the beta function and $\Gamma$ is the gamma function. Moreover, by considering exactly the steps described in \cite{BushellEdmunds2012} for the proof of \cite[(4.15)]{BushellEdmunds2012},
it follows that
\[
    \begin{aligned} 
     \sum_{j=1}^\infty a_j(p,q)&=\frac{\sqrt{2}}{\pi} \int_0^1 \log \left[\cot \left(\frac{\pi x}{2 \pi_{p,q}}
                  \ \hypa\! \! \left(\frac1p,\frac1q;1+\frac1q;x^q\right)\right) \right]\mr{d}x \\
                  &=\frac{\sqrt{2}}{\pi}\int_0^1 \log\left[\cot \left( \frac{\pi}{4} \ \mc{I}\! \!\left(\frac1q,\frac{p-1}{p};x^q\right)   \right)\right] \mr{d}x.
                  \end{aligned}
\]

\begin{figure}
\centerline{
\includegraphics[height=6.4cm, angle=0]{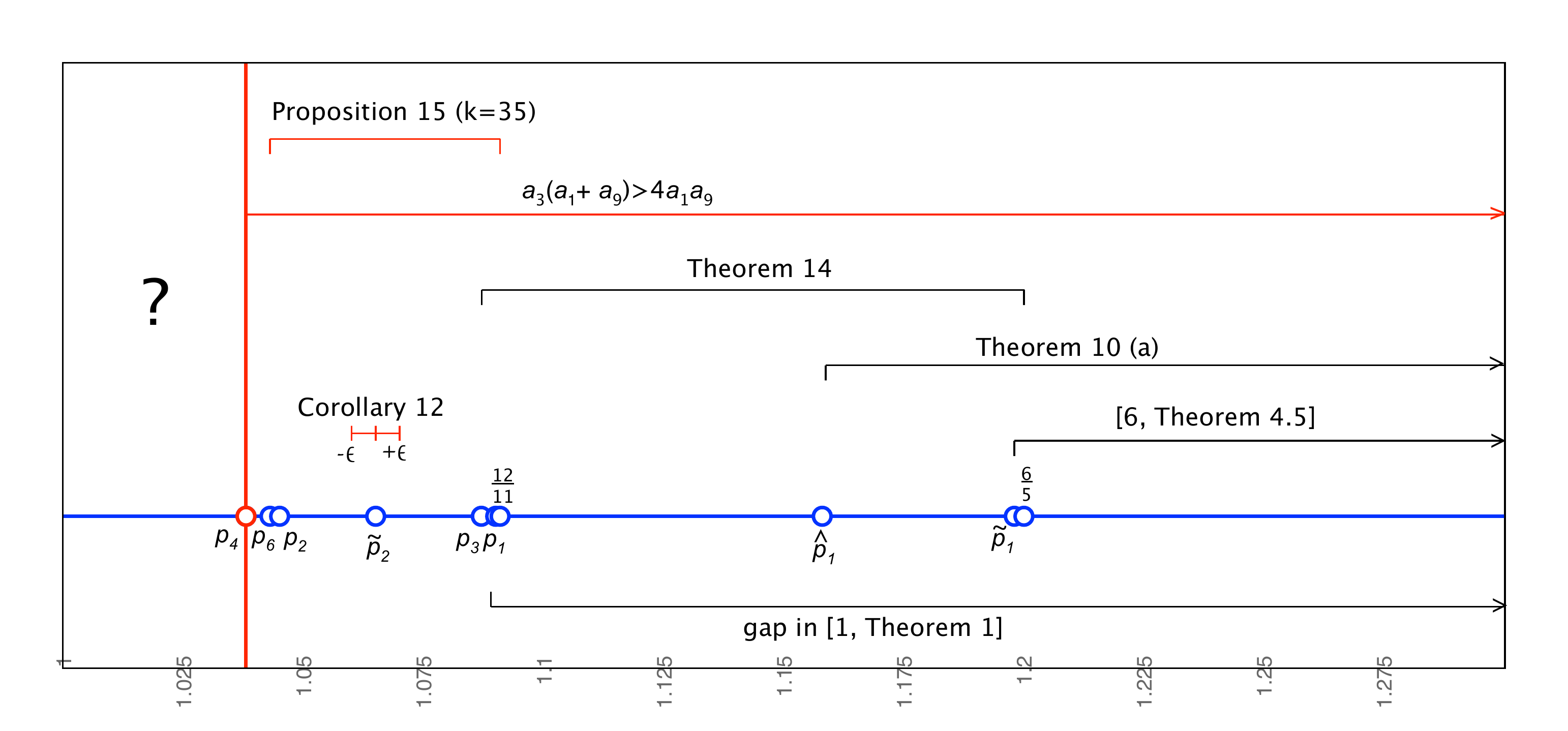}}
\caption{Relation between the various statements of this paper with those of \cite{BBCDG2006} and \cite{BushellEdmunds2012}, for the case $p=q$. The positions of $p_1$, $\tilde{p}_2$ and the value of $\varepsilon$ are set only for illustration purposes, as we are only certain that $p_2<\tilde{p}_2<p_3$.  Black indicates relevance to the general case $r>1$ while red indicates relevance for
the case $r=2$. \label{impro_fig_p=q}}
\end{figure}

Let us begin with the case of equal indices.  See Figure~\ref{impro_fig_p=q}. As mentioned in the introduction, 
\[
      \sum_{k=3}^\infty a_j(p_2,p_2)=a_1(p_2,p_2)
\]
for $p_2\approx 1.043989$. The condition $a_3(p,p)(a_1(p,p)+a_9(p,p))> 4a_9(p,p)a_1(p,p)$ is fulfilled for all $p_4<p<\frac{12}{11}$ where $p_4\approx 1.038537$. The Fourier coefficients $a_j(p,p)\geq 0$ for all $1\leq j \leq 35$ whenever $1<p<\frac{12}{11}$. Remarkably we need to get to $k=35$, for a numerical verification of the conditions of Proposition~\ref{beyond2}
allowing $p<p_2$. Indeed we remark the following. 
\begin{enumerate}
 \item For $k=3,...,33$, the condition \eqref{verygeneral} hold true only for $p_5<p<\frac{12}{11}$ where $p_5\geq 1.044573>p_2$.
\item For $k=35$ the condition \eqref{verygeneral} does hold true for $p_6<p<\frac{12}{11}$ where $p_6\approx 1.043917<p_2$.
\end{enumerate}
This indicates that that the threshold for invertibility of $A$ in the
Hilbert space setting for $p=q$ is at least $p_6$.

Now we examine the general case.
The graphs shown in Figures~\ref{th10ab} and  \ref{th10c} correspond
to regions in the $(p,q)$-plane near $(p,q)=(1,1)$.  Curves on Figure~\ref{th10ab} that are in red are relevant only to the Hilbert space setting 
$r=2$. Black curves pertain to $r>1$. 

Figure~\ref{th10ab}-\emph{(a)} and a blowup shown in Figure~\ref{th10ab}-\emph{(b)}, have two solid (black) lines. One that shows  the limit of applicability of
Theorem~\ref{impro_implicit}-\ref{aimplicit} and one that shows the limit of applicability of the
result of \cite{EdmundsGurkaLang2012}. The dashed line indicates where
\eqref{break2} occurs. To the left of that curve \eqref{trick1} is not applicable.
There are two filled regions of different colours in \emph{(a)}, which indicate where
$a_3(a_1+a_9)< 4a_1a_9$  and where $a_j<0$ for $j=3,9$. 
Proposition~\ref{beyond2} is not applicable in the union of these regions. We also show the lines where $a_3=0$ and $a_9=0$. The latter forms part of the
boundary of this union. The solid red line corresponding to the limit of applicability of Theorem~\ref{impro_implicit}-\ref{bimplicit} 
is also included in  Figure~\ref{th10ab}-\emph{(a)--(d)}. 
To the right of that line, in
the white area, we know that $A$ is invertible for $r=2$. The blowup
in Figure~\ref{th10ab}-\emph{(b)} clearly shows the gap between
Theorem~\ref{impro_implicit}-\ref{aimplicit} and
Theorem~\ref{impro_implicit}-\ref{bimplicit} in this $r=2$ setting.

Certainly $p=q=2$ is a point of intersection for all curves where
$a_j=0$ for $j>1$. These curves are shown in Figure~\ref{th10ab}-\emph{(c)}
also for $j=5$ and $j=7$.  In this figure, we also include 
the boundary of the region where $a_3(a_1+a_9)< 4a_1a_9$ and the region where $a_j<0$ now for $j=3,5,7,9$.
Note that the curves for $a_7=0$ and $a_9=0$ form part of the boundary of the latter. Comparing \emph{(a)} and \emph{(c)}, the new line that cuts
the $p$ axis at $p\approx 1.1$ corresponds to the limit of where
Proposition~\ref{beyond2} for $k=7$ is applicable (for $p$ to the
right of this line).
The gap between the two red lines (case $r=2$) indicates that
Proposition~\ref{beyond2} can significantly improve the threshold for basisness
with respect to a direct application of
Theorem~\ref{impro_implicit}-\ref{bimplicit}. 

As we increase $k$, the boundary of the corresponding region moves to
the left, see the blowups in Figure~\ref{th10ab}-\emph{(d)} and \emph{(e)}. The two 
further curves in red located very close to the vertical axis, 
correspond to the precise value of the parameter $k$ where
Proposition~\ref{beyond2} allows a proof of invertibility for the
change of coordinates which includes the break made by \eqref{break2}. For
$k<35$ the region does not include the dashed black line, for $k=35$
it does include this line.   
The region shown in blue indicates a possible place where
Corollary~\ref{main_1} may still apply, but further investigation in
this respect is needed.

Figure~\ref{th10c} concerns the statement of
Theorem~\ref{impro_implicit}-\ref{cimplicit}. The small wedge shown in
green is the only place where the former is applicable. As it turns,
it appears that the conditions of Corollary~\ref{main_2}  prevent it
to be useful for determining invertibility of $A$ in a neighbourhood
of $(p,q)=(1,1)$. However in the region shown in green, the upper
bound on the Riesz constant consequence of \eqref{trick3} is sharper
than that obtained from \eqref{trick1}.

\appendix

\section{The shape of $\sin_{p,p}$ as $p\to1$} \label{ap1}

Part of the difficulties for a proof of basisness for the family $\mc{S}$ in the regime $p=q\to1$ 
has to do the fact that the Fourier coefficients of $s_1$ approach those of  the function 
$\on{sgn} (\sin(\pi x))$. In this appendix we show that, indeed 
\begin{equation}    \label{regimpto1}
        \lim_{p\to 1} \left(\max_{0\leq x \leq 1}\left|   s_1(x)-\on{sgn} (\sin(\pi x))   \right|\right)
       =0.
\end{equation}

\begin{proof}
Note that 
\[
      \frac{\mr{d}^n}{\mr{d}y^n}s_1^{-1}(y)> 0 \qquad \qquad \forall 0<y<1, \quad n=0,1,2.
\]

Let $y_1(p)\in(0,1)$ be the (unique) value, such that
\[
       \frac{\mr{d}}{\mr{d}y}s_1^{-1}(y_1(p))= \frac{1}{(1-y_1(p)^p)^{1/p}}=\pi_{p,p}.
\] 
Then
\[
     y_1(p)=\left(1-\frac1{\pi_{p,p}^p}\right)^{1/p} \to 1 \qquad p\to 1.
\]
Let 
\[
h(t)=1-\frac{t}{y_1(p)}
\]
be the line passing through the points $(0,1)$ and $(y_1(p),0)$. There exists a unique value $y_2(p)\in(0,y_1(p))$ such that
\[
       \frac{1}{\pi_{p,p}}\frac{1}{(1-y_2(p)^{p})^{1/p}}=h(y_2(p)).
\]
This value is unique because of monotonicity of both sides of this equality, and it exists by bisection.
As all the functions involved are continuous in $p$,  then also $y_2(p)$ is continuous in the parameter  $p$. Moreover,
\[
      y_2(p)\to 1 \qquad p\to 1.
\]
Indeed, by clearing the equation defining $y_2(p)$, we get
\[
      \left( 1-\frac{y_2(p)}{y_1(p)}    \right)^{p}\left( 1-y_2(p)^p  \right)=\frac{1}{\pi_{p,p}^p}.
\]
The right hand side, and thus the left hand side, approach 0 as $p\to 1$. Then, one (and hence both) of the two terms multiplying on the left should approach 0. 

Let $\mf{P}_p$ be the polygon which has as vertices (ordered clockwise)
\begin{gather*}
     v_1(p)=\left(0,\frac{1}{\pi_{p,p}}\right),\quad v_2(p)=(y_2(p),h(y_2(p))), 
 \\     v_3(p)=(y_1(p),1),\quad v_4(p)=(y_1(p),0) \quad \text{and} \qquad v_5= (0,0).
\end{gather*}
As
\begin{gather*}
     v_1(p)\to (0,0)=v_5,\quad v_2(p)\to (1,0), \\
      v_3(p)\to (1,1) \quad \text{and} \quad  v_4(p)\to (1,0);
\end{gather*}
$\mf{B}_p\to ([0,1]\times\{0\}) \cup (\{1\}\times [0,1])$ in Hausdorff 
distance. Then the area of $\mf{P}_p$ approaches 0 as $p\to 1$. Moreover,  $\mf{P}_p$ covers the graph of
\[
    \frac{1}{\pi_{p,p}}\frac{1}{(1-t^{p})^{1/p}}
\] 
for $0<t<y_1(p)$. Thus
\[
    x_1(p)= s_1^{-1}(y_1(p))=\frac{1}{\pi_{p,p}}\int_0^{y_1(p)} \frac{\mr{d} t}{(1-t^p)^{1/p}} \to 0 \qquad p\to 1.
\]
Hence, there is a point $(x_1(p),y_1(p))$ on the graph of $s_1(x)$ such that $0<x_1(p)<1/2$ and
\[
    (x_1(p),y_1(p))\to (0,1).
\]  
The proof of \eqref{regimpto1} is completed from the fact that, as $s_1(x)$ is concave (because its inverse function is convex), the piecewise linear interpolant of $s_1(x)$ for the family of nodes $\{0,x_1(p),1/2\}$ has a graph below that of $s_1(x)$.  
\end{proof}

\section{Basic computer codes} \label{ap2}

The following computer codes written in the open source languages Octave and Python can be used to verify any of the numerical estimations presented in this paper.

\subsection{Octave}
Function for computing $a_j$ with 10-digits precision.
\begin{verbatim}
# -- Function File: [a,err,np]=apq(k,p,q)  
# a is the kth Fourier coefficient of the p,q sine function
# err is the residual
# np number of quadrature points
#
function [a,err,np]=apq(k,p,q)
if mod(k,2)==0, 
   disp('Error: k should be odd');
   return; 
end
[I, err, np]=quadcc(@(y) cos(k*pi*betainc(y.^q,1/q,(p-1)/p)/2),0,1,1e-10);
a=I*2*sqrt(2)/k/pi;
\end{verbatim}

Function for computing $\sum_{j=1}^\infty a_j$ with 10-digits precision.
\begin{verbatim}
# -- Function File: [s,err,np]=apqsum(k,p,q)  
# s is the sum of the Fourier coefficient of the p,q sine function
# err is the residual
# np number of quadrature points
#
function [s,err,np]=apqsum(p,q)
[I, err, np]=quadcc(@(y) log(cot(pi*betainc(y.^q,1/q,(p-1)/p)/4)),0,1,1e-10);
s=I*sqrt(2)/pi;
\end{verbatim}

\subsection{Python - mpmath}
Function for computing $a_j$ with variable precision.
\begin{verbatim}
def a(k,p,q):
    """ Computes the kth Fourier coefficient of the p,q sine function. 
    Returns coefficient and residual.
        >>> from sympy.mpmath import *
        >>> mp.dps = 25; mp.pretty = True
        >>> a(1,mpf(12)/11,mpf(12)/11)
        >>> (0.8877665848468607372062737, 1.0e-59)
    """
    if isint(fraction(k,2)): 
       apq=0;
       E=0;
       return apq,E
    f= lambda x:cos(k*pi*betainc(1/q,(p-1)/p,0,x**q,regularized=True)/2);
    (I,E)=quad(f,[0,1],error=True,maxdegree=10);
    apq=I*2*sqrt(2)/k/pi;
    return apq,E
\end{verbatim}

Function for computing $\sum_{j=1}^\infty a_j$ with variable precision.
\begin{verbatim}
def suma(p,q):
    """ Computes the sum of the Fourier coefficient of the p,q sine function. 
    Returns sum and residual.
        >>> from sympy.mpmath import *
        >>> mp.dps = 25; mp.pretty = True
        >>> suma(mpf(12)/11,mpf(12)/11)
        >>> (1.48634943002852603038783, 1.0e-56)
    """    
    f= lambda x:log(cot(pi*betainc(1/q,(p-1)/p,0,x**q,regularized=True)/4));
    (I,E)=quad(f,[0,1],error=True,maxdegree=10);
    sumapq=I*sqrt(2)/pi;
    return sumapq,E
\end{verbatim}

\section*{Acknowledgements}
The authors wish to express their gratitude to Paul Binding who suggested this problem a few years back. They are also kindly grateful with Stefania Marcantognini for her insightful comments during the preparation of this manuscript. We acknowledge support by the British Engineering and Physical Sciences Research Council  (EP/I00761X/1), the Research Support Fund of the Edinburgh Mathematical Society and the Instituto Venezolano de Investigaciones Cient{\'\i}ficas.
\bibliographystyle{siam}

\begin{figure}
\centerline{\emph{(a)} \hspace{6cm} \emph{(b)}}
\centerline{
\includegraphics[height=6cm, angle=0]{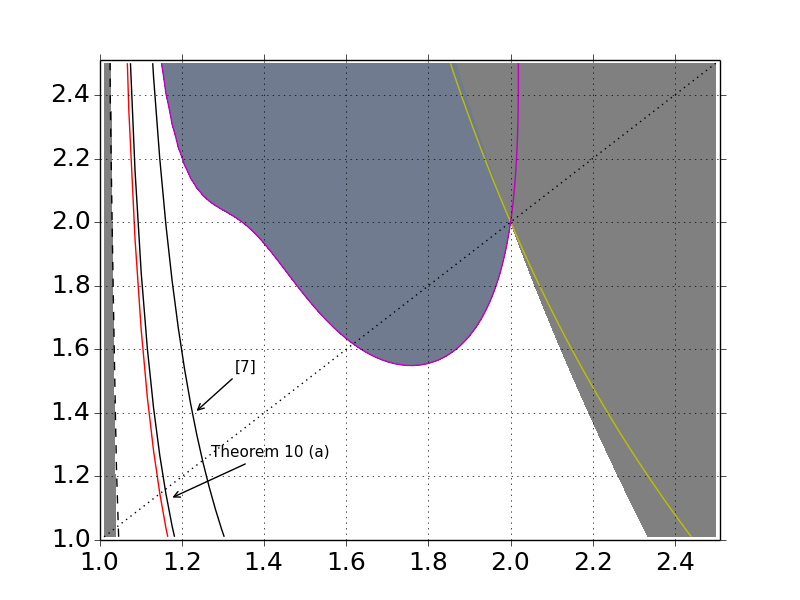} \hspace{-6mm}
\includegraphics[height=6cm, angle=0]{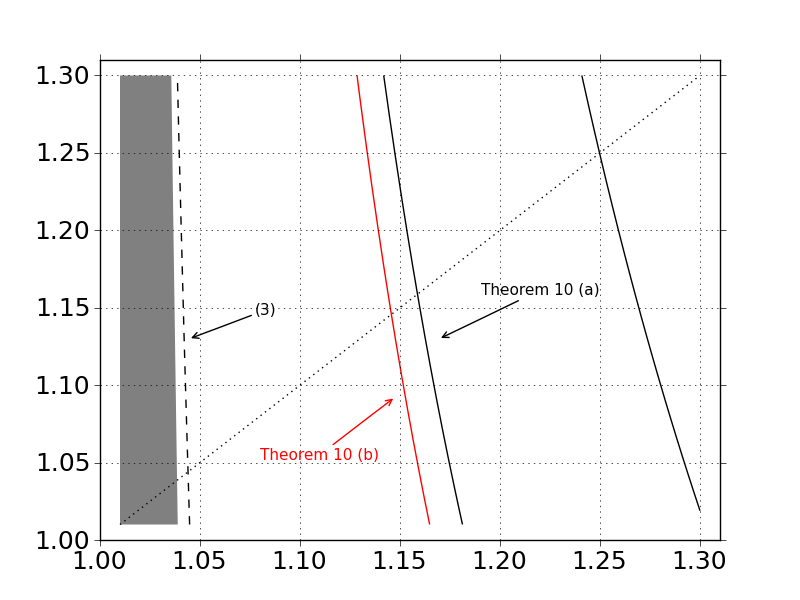}}
\hspace{-6cm}\emph{(c)} 

\centerline{
\includegraphics[height=6cm, angle=0]{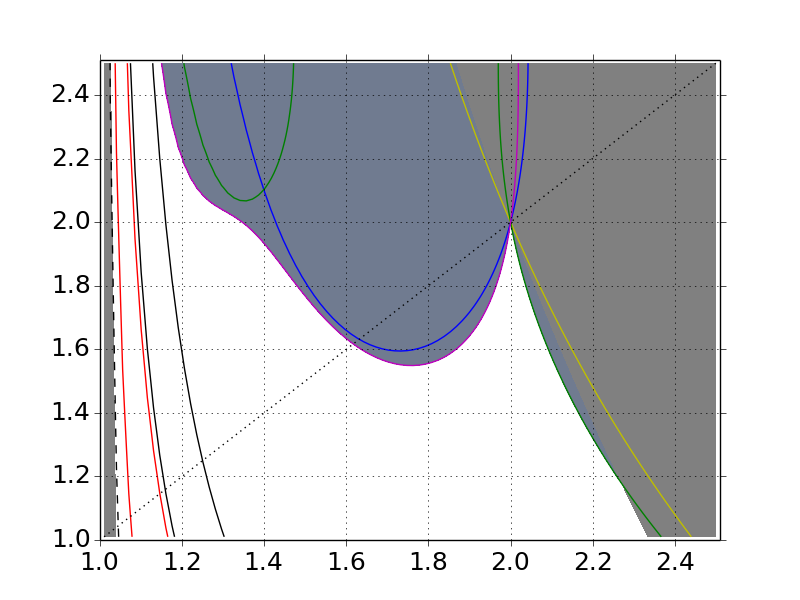}  \hspace{-6mm}
\includegraphics[height=6cm, angle=0]{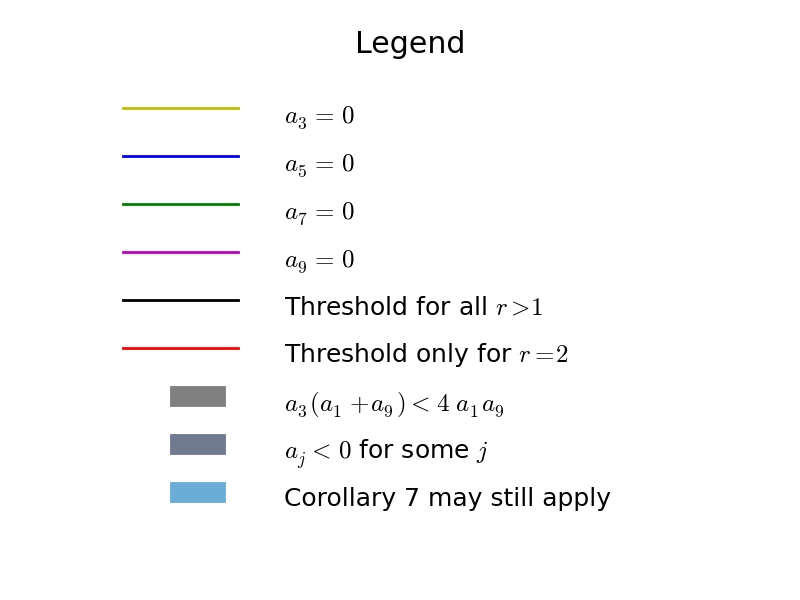}}
\centerline{
\emph{(d)} \hspace{6cm} \emph{(e)}}
\centerline{
\includegraphics[height=6cm, angle=0]{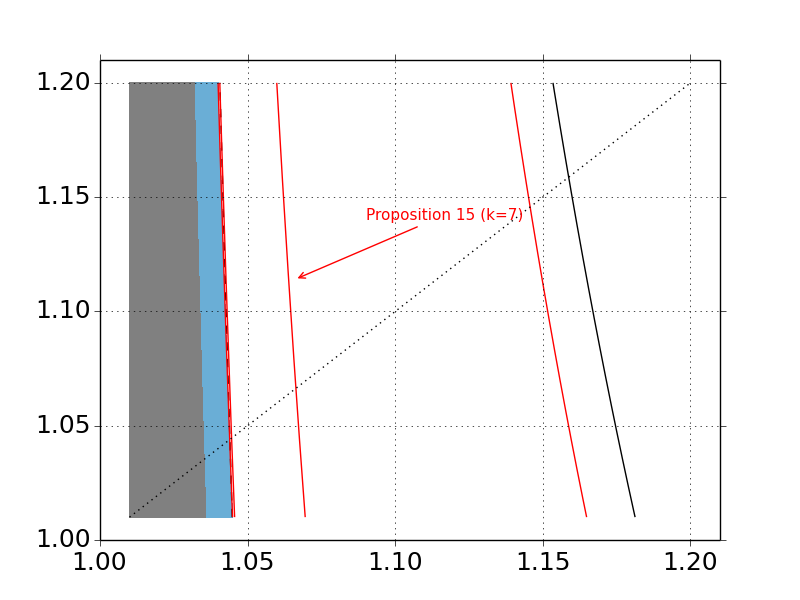} \hspace{-6mm}
\includegraphics[height=6cm, angle=0]{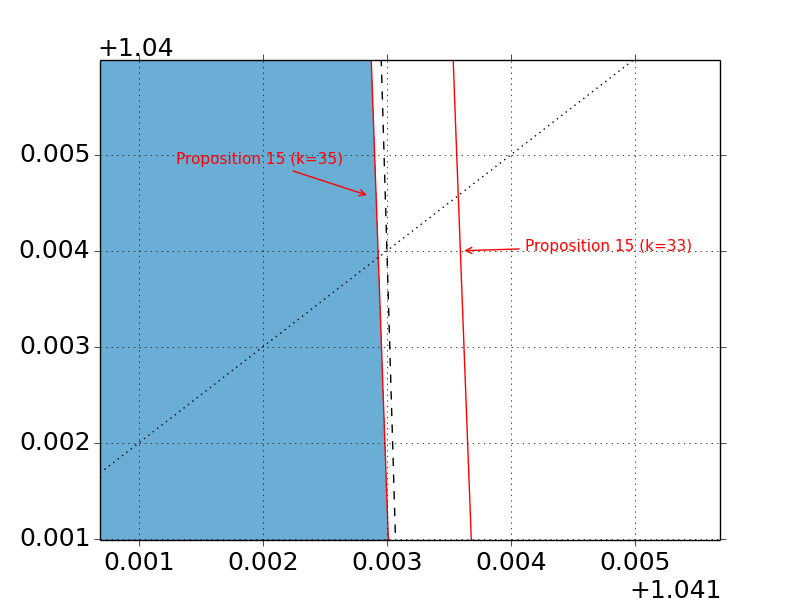}}

\caption{Different relations and boundaries between the regions of the
  $(p,q)$-plane where Theorem~\ref{impro_implicit}-\ref{aimplicit} and
  \ref{bimplicit}, as well as Proposition~\ref{beyond2} (with
  different values of $k$) apply. In all graphs $p$ corresponds to the
  horizontal axis and $q$ to the vertical axis and the dotted line shows $p=q$. \label{th10ab}}
\end{figure}

\begin{figure}
\centerline{
\includegraphics[height=6cm, angle=0]{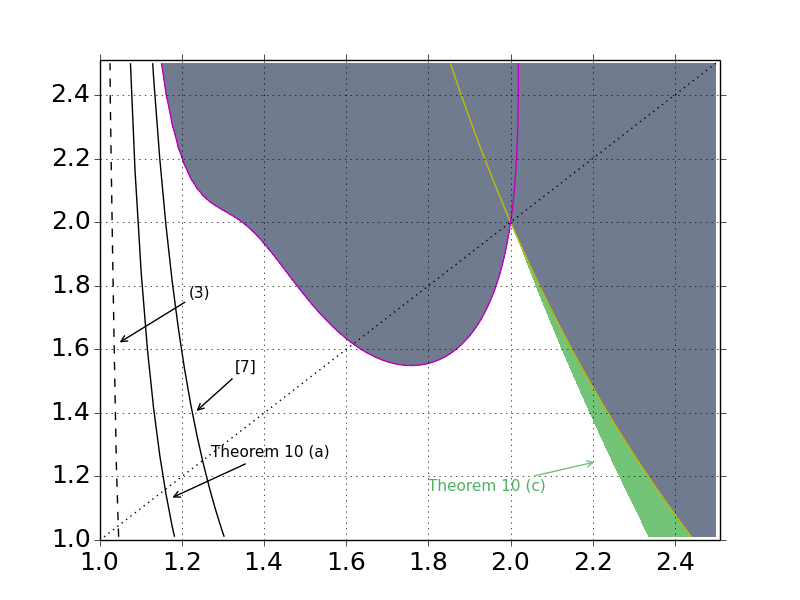}\hspace{-6mm}
\includegraphics[height=6cm, angle=0]{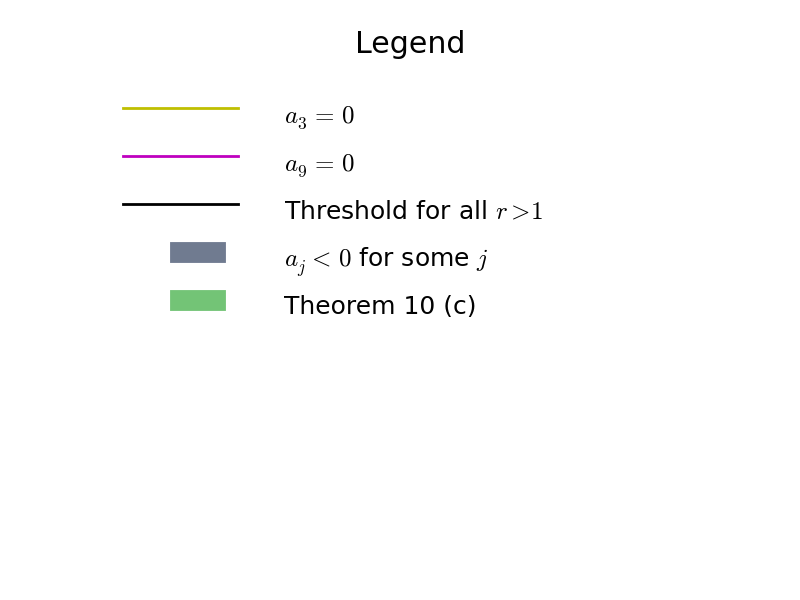}}
\caption{Region of the $(p,q)$-plane where
  Theorem~\ref{impro_implicit}-\ref{cimplicit} applies. Even when we
  know $A$ is invertible in this region as a consequence of
  Theorem~\ref{impro_implicit}-\ref{aimplicit}, the upper bound on the
  Riesz constant  provided by \eqref{trick3} improves upon that
  provided by \eqref{trick1} (case $r=2$). In this graph $p$
  corresponds to the horizontal axis and $q$ to the vertical axis and the dotted line shows $p=q$.  \label{th10c}}
\end{figure}

\end{document}